\documentclass[preprint]{elsarticle}
\usepackage{tikz}
\usepackage[english]{babel}
\usepackage{amssymb,latexsym}
\usepackage{amsmath,amsfonts}
\usepackage{amsthm}
\usepackage{url}
\usepackage{enumerate}
\usepackage{geometry}
\usepackage{url}
\usepackage{microtype}
\usepackage{color}

\newcommand{\KK}{\mathbb{K}}

\newcommand{\cL}{\mathcal{L}}

\newcommand{\fS}{\mathfrak{S}}

\newcommand{\Gr}{\mathrm{Gr}}

\newcommand{\cP}{\mathcal{P}}

\newcommand{\LL}{\mathbb{L}}

\newcommand{\PG}{\mathrm{PG}}

\newcommand{\FF}{\mathbb{F}}

\newcommand{\gr}{\mathrm{gr}}
\newcommand{\er}{\mathrm{er}}

\theoremstyle{plain}

\newtheorem{lemma}{Lemma}[section]
\newtheorem{theorem}[lemma]{Theorem}
\newtheorem{corollary}[lemma]{Corollary}
\newtheorem{proposition}[lemma]{Proposition}

\theoremstyle{definition}
\newtheorem{definition}[lemma]{Definition}
\newtheorem{remark}[lemma]{Remark}

\begin{document}
\begin{frontmatter}

  \title{On the generation of some Lie-type geometries}
  \author[SI]{Ilaria Cardinali}
\ead{ilaria.cardinali@unisi.it}
\address[SI]{Department of Information Engineering and Mathematics, University of Siena,
Via Roma 56, I-53100, Siena, Italy}
\author[LG]{Luca Giuzzi\corref{cor1}}
\ead{luca.giuzzi@unibs.it}
\address[LG]{DICATAM - Section of Mathematics,
University of Brescia,
Via Branze 53, I-25123, Brescia, Italy}
\cortext[cor1]{Corresponding author. Tel. +39 030 3715739; Fax. +39 030 3615745}
\author[SI,r]{Antonio Pasini}
\ead{antonio.pasini@unisi.it}
\fntext[r]{(retired)}
\begin{abstract}
Let $X_n(\KK)$ be a building of Coxeter type $X_n = A_n$ or $X_n = D_n$ defined over a given division ring $\KK$ (a field when $X_n = D_n$). For a non-connected set $J$ of nodes of the diagram $X_n$, let $\Gamma(\KK) = \Gr_J(X_n(\KK))$ be the $J$-grassmannian of $X_n(\KK)$. We prove that $\Gamma(\KK)$ cannot be generated over any proper sub-division ring $\KK_0$ of $\KK$. As a consequence, the generating rank of $\Gamma(\KK)$ is infinite when $\KK$ is not finitely generated. In particular, if $\KK$ is the algebraic closure of a finite field of prime order then the generating rank of $\Gr_{1,n}(A_n(\KK))$ is infinite, although its embedding rank is either $(n+1)^2-1$ or $(n+1)^2$.
\end{abstract}
\begin{keyword}
  Coxeter building \sep $A_n$ \sep $D_n$ \sep
  Grassmann geometry \sep Subgeometry \sep Generation
  \MSC[2010] 51E24 \sep 51M30 \sep 51M35 \sep 14M15
\end{keyword}
\end{frontmatter}
\section{Introduction}\label{Introduction}

We presume that the reader has some acquaintance with buildings and is familiar with basics of point-line geometry. In case, we refer to Tits \cite{Tits} for buildings and Shult \cite{SH11} for point-line geometries. We only recall the notion of generation in point-line geometries. A {\em subspace} of a point-line geometry $\Gamma = (\cP, \cL)$ is a subset $S$ of the point-set $\cP$ of $\Gamma$ such that, for every line $\ell\in \cL$, if $|S\cap\ell| > 1$ then $\ell\subseteq S$. We say that a subset $X$ of a subspace $S$ {\em generates} $S$, in symbols $S = \langle X\rangle$, if $S$ is the minimum subspace of $\Gamma$ containing $X$, namely $S$ is the intersection of all subspaces of $\Gamma$ which contain $X$. We also  recall that the \emph{generating rank} of a point-line geometry $\Gamma=(\cP,\cL)$ is the number $\gr(\Gamma):=\min \{|X|\colon X\subseteq\cP, \langle X\rangle= \cP\}.$

\subsection{Basic definitions and known results}\label{Intro1}
Let $\mathfrak{X}$ be a class of buildings such that, for every division ring $\KK$, at most one (up to isomorphism) member of $\mathfrak{X}$ is defined over $\KK$. For instance, $\mathfrak{X}$ can be the class of buildings belonging to a given simply laced Coxeter diagram or a given Dynkin diagram, possibly of twisted type. With $\mathfrak{X}$ as above, let $\Delta(\KK)$ be the member of $\mathfrak{X}$ defined over $\KK$ (provided it exists) and, for a nonempty subset $J$ of the type-set of $\Delta(\KK)$, let $\Gr_J(\Delta(\KK))$ be the $J$-grassmannian of $\Delta(\KK)$, regarded as a point-line geometry. For a sub-division ring $\KK_0$ of $\KK$, suppose that $\mathfrak{X}$ also contains a member $\Delta(\KK_0)$ defined over $\KK_0$ and $\Gr_J(\Delta(\KK))$ contains $\Gr_J(\Delta(\KK_0))$ as a subgeometry (as it is always the case for the geometries to be considered in this paper). We say that $\Gr_J(\Delta(\KK))$ is {\em generated over} $\KK_0$ ($\KK_0$-{\em generated} for short) if $\Gr_J(\Delta(\KK_0))$, as a set of points of the point-line geometry $\Gr_J(\Delta(\KK))$, generates $\Gr_J(\Delta(\KK))$.

Clearly, if $\Gr_J(\Delta(\KK))$ is $\KK_0$-generated and $\Gr_J(\Delta(\KK_0))$ is $\KK_1$-generated for a division ring $\KK_1 < \KK_0$, then $\Gr_J(\Delta(\KK))$ is $\KK_1$-generated too. It is also clear that if $\Gr_J(\Delta(\KK))$ is $\KK_0$-generated then the generating rank of $\Gr_J(\Delta(\KK))$ cannot be larger than that of $\Gr_J(\Delta(\KK_0))$. On the other hand, suppose that every finite set of points of $\Gr_J(\Delta(\KK))$ belongs to a subgeometry of $\Gr_J(\Delta(\KK))$ isomorphic to $\Gr_J(\Delta(\KK_0))$ for a finitely generated sub-division ring $\KK_0$ of $\KK$ (as it is often the case). Suppose moreover that $\KK$ is not finitely generated as a division-ring and that $\Gr_J(\Delta(\KK))$ is not $\KK_0$-generated, for any  $\KK_0 < \KK$. Then $\Gr_J(\Delta(\KK))$ has infinite generating rank, as we prove in Lemma \ref{cor-inf}. In short, obvious links exist between the $\KK_0$-generation problem and the computation of generating ranks. Less obviously, some relations also seem to exist between $\KK_0$-generability and the existence of the absolutely universal embedding. For instance, a number of grassmannians $\Gr_J(\Delta(\KK))$ for which the existence of the absolutely universal embedding is still an open problem, cannot be generated over any proper sub-division ring $\KK_0$ of $\KK$ (see Section \ref{ranghi}, Remark \ref{Rem1.6}).

We shall now briefly summarize what is currently known about $\KK_0$-generation. For $X_n$ a simply laced Coxeter diagram of rank $n$ or a Dynkin diagram of rank $n$ (but not of twisted type) and a division ring $\KK$ (a field if $X_n\neq A_n$), let $X_n(\KK)$ be the unique building of type $X_n$ defined over $\KK$. In particular, $B_n(\KK)$ and $C_n(\KK)$ are the buildings associated to the orthogonal group $\mathrm{O}(2n+1,\KK)$ and the symplectic group $\mathrm{Sp}(2n,\KK)$ respectively.

Suppose firstly that $\Gr_J(X_n(\KK))$ is spanned by $\Gr_J(A)$ for an apartment $A$ of $X_n(\KK)$ (for short, $\Gr_J(X_n(\KK))$ {\em is spanned by an apartment}). For every sub-division ring $\KK_0$ of $\KK$, the geometry $\Gr_J(A)$ is contained in a subgeometry of $\Gr_J(X_n(\KK))$ isomorphic to $\Gr_J(X_n(\KK_0))$. Hence $\Gr_J(X_n(\KK))$ is $\KK_0$-generated for any $\KK_0 \leq \KK$. In particular, $\Gr_J(X_n(\KK))$ is generated over the prime subfield of $\KK$.

It is known (Cooperstein and Shult~\cite{CS97}, Blok and Brouwer~\cite{BB98}) that the following grassmannians are generated by apartments, where
we take the integers $1, 2,\dots, n$ as types as usual but when $X_n = D_n$, according to the notation adopted in Section~\ref{mres}, we replace $n-1$ and $n$ with $+$ and $-$: $\Gr_k(A_n(\KK))$ for $1\leq k\leq n$; $\Gr_1(D_n(\KK))$ and $\Gr_+(D_n(\KK))$ as well as $\Gr_-(D_n(\KK))$; $\Gr_1(C_n(\KK))$ and $\Gr_n(B_n(\KK))$ but with $\mathrm{char}(\KK)\neq 2$ in both cases; $\Gr_1(E_6(\KK))$, $\Gr_6(E_6(\KK)))$ and $\Gr_1(E_7(\KK))$ (the nodes of the $E_7$-diagram being labeled as in~\cite{CS97}). Therefore, all above mentioned grassmannians are generated over the prime subfield of $\KK$. It is easily seen that the same holds for $\Gr_1(B_n(\KK))$, even if this geometry is not spanned by any apartment. It is likely that if $\mathrm{char}(\KK) \neq 2$ then, for every $i \leq n$, the $i$-grassmannian $\Gr_i(C_n(\KK))$ is generated over the prime subfield of $\KK$, but we are not aware of any explicit proof of this claim.

We now turn to $\Gr_{1,n}(A_n(\KK))$. This geometry is interesting in its own right. When $\KK$ is a field it is known as the {\em long root geometry} for $\mathrm{SL}(n+1,\KK)$. In~\cite{BPa01} it is proved that if $n > 2$ then $\Gr_{1,n}(A_n(\KK))$ is not $\KK_0$-generated, for any proper sub-division ring $\KK_0$ of $\KK$ (see also~\cite[Theorem 5.10]{ILP19} for an alternative proof in the special case where $n = 3$ and $\KK$ is a field). However, when $\KK$ is a field and $\KK=\KK_0(a_1,\dots, a_t)$ for suitable elements $a_1,\dots, a_t\in \KK\setminus \KK_0$, then $\Gr_{1,n}(A_n(\KK))$ can be generated by adding at most $t$ elements to $\Gr_{1,n}(A_n(\KK_0))$ (Blok and Pasini~\cite{BPa01}). In particular, when $\KK$ is finite, $(n+1)^2$ points are enough to generate  $\Gr_{1,n}(A_n(\KK_0))$. Indeed, in this case $\KK$ is a simple extension of its prime subfield $\KK_0$ and the generating rank of $\Gr_{1,n}(A_n(\KK_0))$ is equal to $(n+1)^2-1$ (Cooperstein~\cite{C98b}).

Not much is known on $\Gr_k(B_n(\KK))$ for $1 < k < n$ and $\Gr_k(D_n(\KK))$ for $1 < k \leq n-2$. Probably, what makes these cases so difficult to investigate is the fact that the special case  $\Gr_{1,3}(A_3(\KK))\cong \Gr_{+,-}(D_3(\KK))$ of $\Gr_{1,n}(A_n(\KK))$ somehow enters the game in any attempt to compute the generating rank of $\Gr_k(B_n(\KK))$ or $\Gr_k(D_n(\KK))$ and, as we have seen above, as far as generation is concerned, $\Gr_{1,n}(A_n(\KK))$ can behave wildly. Nevertheless,
in~\cite{ILP19} we have shown that for $\KK=\FF_4,\FF_8$ or $\FF_9$ the grassmannians $\Gr_2(B_n(\KK))$ ($n\geq 3$) and $\Gr_2(D_n(\KK))$ ($n > 3$) are generated over the corresponding prime subfields $\FF_2$ or $\FF_3$. The generating ranks of $\Gr_2(B_n(\KK_0))$ and $\Gr_2(D_n(\KK_0))$, for $\KK_0$ a finite field of prime order, are known to be equal to ${{2n+1}\choose 2}$ and ${{2n}\choose 2}$ respectively (Cooperstein \cite{C98b}). Hence ${{2n+1}\choose 2}$ and ${{2n}\choose 2}$ are the generating ranks of $\Gr_2(B_n(\KK))$ and $\Gr_2(D_n(\KK))$ respectively, with $\KK$ as above.

\subsection{Main results}

We state in this subsection the main results of this paper. We refer to Section~\ref{prelim} or to \cite{PasDG} and \cite{BuekC} for the notation and further details on the grassmannians we are dealing with.
The following, to be proved in Section \ref{pro}, is our first main result in this paper:
\begin{theorem}
  \label{m-th1}
For a division ring $\KK$, let $\Gamma(\KK)$ be one of the following: $Gr_{{1,n}}(A_n(\KK))$ for $n\geq 3$; $\Gr_{{+,-}}(D_n(\KK))$, $n\geq 3$;  $\Gr_{{1,+,-}}(D_n(\KK))$ with $n\geq 4$; $\Gr_{{1,-}}(D_n(\KK))$ for $n\geq 4$. Then $\Gamma(\KK)$ is not
$\KK_0$-generated for any proper sub-division ring $\KK_0$ of $\KK$.
 \end{theorem}
As said in Section \ref{Intro1}, the case of $\Gr_{{1,n}}(A_n(\KK))$ has been already considered in~\cite{BPa01}, but the proof we shall give in this paper is different and simpler than that of \cite{BPa01}. Theorem \ref{m-th1} also contains a proof of a  conjecture presented in~\cite[Remark 5.11]{ILP19}.
\begin{corollary}\label{vecchio}
The $(n-1)$-grassmannian $\Gr_{n-1}(B^+_n(\KK))$ of the top-thin polar space $B^+_n(\KK) = \Gr_1(D_n(\KK))$ is not $\KK_0$-generated for any proper subfield $\KK_0$ of $\KK$.
\end{corollary}
Seemingly, this corollary is an obvious consequence of Theorem \ref{m-th1} and the isomorphism $\Gr_{n-1}(B^+_n(\KK)) \cong \Gr_{+,-}(D_n(\KK))$. However its proof is not so trivial as one might believe; we will give it in Section \ref{subtle}.
As we shall see in Section \ref{Sec 3 bis}, Theorem \ref{m-th1} admits the following far reaching generalization:
\begin{theorem}
  \label{non-interval}
Let $\Gamma(\KK)$ be either $\Gr_{J}(A_n(\KK))$ or $\Gr_{J}(D_n(\KK))$, with $J$ a non-connected set of nodes of the diagram $A_n$ or $D_n$ respectively. Then $\Gamma(\KK)$ is not $\KK_0$-generated, for any proper sub-division ring $\KK_0$ of $\KK$.
\end{theorem}

\subsection{Applications to generating ranks and embeddings}\label{ranghi}

Given a point-line geometry  $\Gamma=(\cP,\cL),$ a ({\em full}) {\em projective embedding} $e:\Gamma\rightarrow\PG(V)$ of $\Gamma$ (henceforth often called simply an {\em embedding} of $\Gamma$, for short) is an injective map $e:\cP\to\PG(V)$ from the point-set $\cP$ of $\Gamma$ to the set of points of the projective space $\PG(V)$ of a vector space $V$, such that for every line $\ell\in{\mathcal L}$ of $\Gamma$ the set $e(\ell):=\{ e(p)\colon p\in\ell\}$ is a projective
line of $\PG(V)$ and $e(\cP)$ spans $\PG(V)$. We put $\dim(e):=\dim(V)$, calling $\dim(e)$ the {\em dimension} of $e$. If $\KK$ is the underlying division ring of $V$, we say that $e$ is {\em defined over} $\KK$, also that $e$ is a $\KK$-{\em embedding}. If $\Gamma$ admits a projective embedding we say that $\Gamma$ is {\em projectively embeddable} (also {\em embeddable}, for short).
If $e:\Gamma\to\PG(V)$ and $e':\Gamma\to\PG(V')$ are two $\KK$-embeddings of $\Gamma$ we say that $e$ \emph{dominates} $e'$ if there is
a $\KK$-semilinear mapping $\varphi:V\rightarrow V'$ such that $e' =\varphi\cdot e$. If $\varphi$ is an isomorphism then we say that $e$ and $e'$ are {\em isomorphic}. Following Tits~\cite{Tits}, we say that an embedding $e$ is \emph{dominant} if, modulo isomorphisms, it is not dominated by any embedding other than itself. Every $\KK$-embedding $e$ of $\Gamma$ admits a {\em hull} $\tilde{e}$, uniquely determined up to isomorphism and characterized by the following property: $\tilde{e}$ dominates all $\KK$-embeddings of $\Gamma$ which dominate $e$ (see Ronan~\cite{Ronan}). Accordingly, an embedding is dominant if and only if it is the hull of at least one embedding; equivalently, if and only if it is its own hull. Finally, an embedding $\tilde{e}$ of $\Gamma$ is \emph{absolutely universal} (henceforth called just {\em universal}, for short) if it dominates all embeddings of $\Gamma$. In other words, $\Gamma$ admits the universal embedding if and only if all of its embeddings have the same hull, that common hull being the universal embedding of $\Gamma$. Note that this forces all embeddings of $\Gamma$ to be defined over the same division ring. Note also that the universal embedding, if it exists, is homogeneous, an embedding $e$ of $\Gamma$ being {\em homogeneous} if $eg \cong e$ for every automorphism $g$ of $\Gamma$.
The \emph{embedding rank} $\er(\Gamma)$ of an embeddable geometry $\Gamma$ is defined as follows:
  \[ \er(\Gamma):=\sup \{ \dim(\varepsilon): \varepsilon
    \text{ projective embedding of }\Gamma \}. \]
Obviously, if $\Gamma$ admits the universal embedding $\tilde{e}$ then $\er(\Gamma) = \dim(\tilde{e})$, but $\er(\Gamma)$ is defined even if no embedding of $\Gamma$ is universal.
If $e:\Gamma\to\PG(V)$ is an embedding of $\Gamma = (\cP,\cL)$ then stretching a line in $\Gamma$ through two collinear points $p, q\in \cP$ corresponds to forming the span $\langle v, w\rangle \subseteq V$ of any two non-zero vectors $v\in e(p)$ and $w\in e(q)$. If $X\subseteq \cP$ generates $\Gamma$ then $\cP = \cup_{n=0}^\infty X_n$ where $X_0 := X$ and $X_{n+1} := \cup_{p,q\in X_n}\langle p, q\rangle_\Gamma$. Consequently, if we select a non-zero vector $v_p\in e(p)$ for every point $p\in X$ then $\{v_p\}_{p\in X}$ spans $V$. This makes it clear that $|X| \geq \dim(e)$. Accordingly,
\begin{equation}\label{gen emb rk 1}
\dim(e)  ~ \leq ~ \gr(\Gamma).
\end{equation}
Therefore, if $\gr(\Gamma)$ is finite and $\dim(e) = \gr(\Gamma)$ then $e$ is dominant (hence universal, if $\Gamma$ admits the universal embedding). In any case, \eqref{gen emb rk 1} implies the following:
\begin{equation}\label{gen emb rk 2}
\er(\Gamma) ~ \leq ~ \gr(\Gamma).
\end{equation}
In fact the equality $\er(\Gamma)=\gr(\Gamma)$ holds for many embeddable geometries, but not for all of them. For instance Heiss \cite{H00} gives an example where $\gr(\Gamma) = \er(\Gamma) +1$. The example of \cite{H00} looks fairly artificial. A more natural example, where $\er(\Gamma)$ is finite but $\gr(\Gamma)$ is infinite, is given by Theorem~\ref{m-th2}, to be stated below. That theorem will be obtained in Section~\ref{Nuovo} with the help of the following lemma. In order to properly state it, we recall that a division ring $\KK$ is \emph{finitely generated} if it is generated as a division ring by a finite subset $X\subseteq \KK$, namely no proper sub-division ring of $\KK$ contains $X$. For instance, an algebraic extension of a finite field of prime order $\FF_p$ is finitely generated if and only if it is finite, in which case it is a simple extension of $\FF_p$. On the other hand, no algebraically closed field is finitely generated.
\begin{lemma}\label{cor-inf}
Let $\Gamma(\KK)$ be either $\Gr_{J}(A_n(\KK))$ or $\Gr_{J}(D_n(\KK))$ for a set of types $J$ non-connected as a set of nodes of $A_n$ or $D_n$. Suppose that $\KK$ is not finitely generated. Then the generating rank of $\Gamma(\KK)$ is infinite.
\end{lemma}
Lemma~\ref{cor-inf} will be obtained in Section~\ref{Nuovo} as a consequence of Theorem~\ref{non-interval}. By exploiting it we will obtain the following:
\begin{theorem}\label{m-th2}
 Let $\FF_p$ be a finite field of prime order and $\overline{\FF}_p$ its algebraic closure. Then, for $n \geq 3$, the geometry $\Gr_{1,n}(A_n(\overline{\FF}_p))$ has infinite generating rank but its embedding rank is equal to either $(n+1)^2-1$ or $(n+1)^2$.
\end{theorem}
\begin{remark}\label{Rem1.6}
 It is well known that if $\KK$ is a field then $\Gr_{1,n}(A_n(\KK))$ admits an $(n+1)^2-1$ dimensional embedding, say $e_{\mathrm{Lie}}$, in (the projective space of) the space of square matrices of order $n+1$ with entries in $\KK$ and null trace (see e.g. Blok and Pasini~\cite{BPa03}; the choice of the symbol $e_{\mathrm{Lie}}$ for this embedding is motivated by the fact that it affords the representation of the group $\mathrm{SL}(n+1,\KK)$ in its action on its own Lie algebra). However $\Gr_{1,n}(A_n(\KK))$  does not satisfy the sufficient conditions of Kasikova and Shult~\cite{K-S} for the existence of the universal embedding. So, we do not know if it always admits the universal embedding, let alone if $e_{\mathrm{Lie}}$ is universal. A complete answer is known only when $\KK$ is a prime field. In this case $e_{\mathrm{Lie}}$ is indeed universal (Blok and Pasini~\cite[Section 3]{BPa03}). A bit less is known when $\KK$ is a number field or a perfect field of positive characteristic; in this case $e_{\mathrm{Lie}}$ dominates all homogeneous embeddings of $\Gr_{1,n}(A_n(\KK))$ (V\"{o}lklein \cite{Volklein}).
As for the remaining geometries of Theorem \ref{m-th1}, namely $\Gr_{{+,-}}(D_n(\KK))$, $\Gr_{{1,+,-}}(D_n(\KK))$ and $\Gr_{{1,-}}(D_n(\KK))$, they too are embeddable (see~\cite{BPa03}) and, when $\KK$ is a prime field, they admit the universal embedding (Blok and Pasini~\cite[Section 4]{BPa03}), even if none of them satisfies the conditions of Kasikova and Shult~\cite{K-S}.
\end{remark}
\begin{remark}
The geometry $\Delta^+_2$ of~\cite{IP13} with $n = 3$ is the same as $\Gr_{1,3}(A_3(\FF))$. According to the above, Lemma 4.8 of~\cite{IP13}, which deals with that geometry and its Weyl embedding $\varepsilon_2^+$ (which is the same as $e_{\mathrm{Lie}}$), might possibly be wrong as stated. It should be corrected as follows: when $n = 3$ and $\FF$ is a perfect field of positive characteristic or a number field, then $\tilde{\varepsilon}^+_2$ dominates all homogeneous embeddings of $\Delta^+_2$.
\end{remark}
\begin{remark}\label{lax}
In our survey of embeddings we have stuck to full projective embeddings, but in the proof of Theorem \ref{m-th2} we shall deal also with lax embeddings. {\em Lax projective embeddings} are defined in the same way as full projective embeddings but for replacing the condition that $e(\ell)$ is a line of $\PG(V)$ with the weaker condition that $e(\ell)$ spans a line of $\PG(V)$, for every line $\ell$ of $\Gamma$. Many authors also require that no two lines of $\Gamma$ span the same line of $\PG(V)$, but in view of our needs in this paper we can safely renounce that requirement. The only fact relevant for us is that inequality \eqref{gen emb rk 1} holds true even if $e$ is lax, as it is clear from the way we have obtained it.
\end{remark}

 \section{Preliminaries}
\label{prelim}
\subsection{Setting and notation}
\label{mres}
We refer to~\cite[Chapter 5]{PasDG} for the definition of the $J$-grassmannian $\Gr_J(\Delta)$ of a geometry $\Delta$. We recall that when $\Delta$ satisfies the so-called Intersection Property (which is always the case when $\Delta$ is a building) then $\Gr_J(\Delta)$ is the same as the $J$-shadow space of $\Delta$ as defined by Tits~\cite[Chapter 12]{Tits}. According to~\cite{PasDG} (and~\cite{Tits}), the $J$-grassmannian of a geometry $\Delta$ is a geometry with a string-shaped diagram graph and the same rank as $\Delta$, but in this paper, following Buekenhout and Cohen~\cite[\S 2.5]{BuekC}, we shall mostly regard it as a point-line geometry, with the $J$-flags of $\Delta$ taken as points, while the lines are the flags of $\Delta$ of type $(J\setminus\{j\})\cup\mathrm{fr}(j)$ for $j\in J$, where $\mathrm{fr}(j)$ stands for the set of types adjacent to $j$ in the diagram of $\Delta$; a point and a line of $\Gr_J(\Delta)$ are incident precisely when they are incident as flags of $\Delta$.
So, the lines of $\Gr_J(\Delta)$ are particular flags of $\Delta$. This setting will indeed be helpful in some respects but it forces to distinguish between a line and its set of points and this distinction often ends in a burden for the exposition; we will often neglect it. This is a harmless abuse. Indeed only grassmannians of buildings are considered in this paper; buildings satisfy the Intersection Property and, if that property holds in a geometry $\Delta$, then no two lines of $\Gr_J(\Delta)$ have the same points (even better: no two lines of $\Gr_J(\Delta)$ have two points in common).
As in Section~\ref{Intro1}, given a division ring $\KK$, we denote by $A_n(\KK)$ the building of type $A_n$ defined over $\KK$. Similarly, if the division ring $\KK$ is a field (namely, is commutative) then $D_n(\KK)$ stands for the building of type $D_n$ defined over $\KK$. We allow $n = 3$ in $D_n$. So, $D_3 = A_3$. Nevertheless, when writing $D_3(\KK)$ we always understand that $\KK$ is a field, for consistency of notation.
Let $X_n$ stand for either $A_n$ or $D_n$. It is well known that the elements of $X_n(\KK)$ can be identified with suitable vector subspaces of a vector space $V$ over $\KK$ of dimension either $n+1$ or $2n$ according to whether $X_n = A_n$ or $X_n = D_n$. Similarly, given a proper sub-division ring  $\KK_0$ of $\KK$, the building $X_n(\KK_0)$ is realized in a vector space $V_0$ over $\KK_0$, of the same dimension as $V$. We can always assume that $V_0$ is the set of $\KK_0$-linear combinations of the vectors of a selected basis $E$ of $V$, so that $V$ is obtained from $V_0$ by scalar extension from $\KK_0$ to $\KK$. Thus, with $E$ suitably selected when $X_n = D_n$, the building $X_n(\KK_0)$ is turned into a subgeometry of $X_n(\KK)$ (see Sections~\ref{rational1} and \ref{rational2} for more details). Accordingly, for every subset $J$ of the set of nodes of the diagram $X_n$, the $J$-grassmannian $\Gr_J(X_n(\KK_0))$ can be regarded as a subgeometry of $\Gr_J(X_n(\KK))$. Our main goal in this paper is to show that, if $J$ consists of extremal nodes of $X_n$ and $|J| > 1$ then $\Gr_J(X_n(\KK_0))$ does not generate $\Gr_J(X_n(\KK))$.
We firstly consider the $\{1,n\}$-grassmannian $\Gr_{1,n}(A_n(\KK))$ of $A_n(\KK)$; see
Fig.~\ref{Gr1n(An)}.
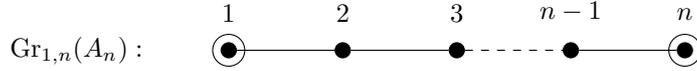
\begin{figure}[h]
\begin{center}
  \begin{tikzpicture}(50,7)
            \tikzstyle{every node}=[draw,circle,fill=black,minimum size=2mm,
            label distance=4pt, inner sep=2pt]
            \node [draw=none,rectangle,fill=none] at (-2,0) {$\Gr_{{1,n}}(A_n):$};
            \draw
            (0,0) node[label=above:$1$](1) {}
            (1.5,0) node[label=above:$2$](2) {}
            (3,0) node[label=above:$3$](3) {}
            (4.5,0) node[label={[yshift=-0.25cm]{$n-1$}}](n-1) {}
            (6,0) node[label={[yshift=+0.0cm]$n$}](n) {};
            \draw (1)--(2)--(3);
            \draw (n-1)--(n);
            \draw [dashed] (3)--(n-1);
            \draw (n) circle (6pt);
            \draw (1) circle (6pt);
\end{tikzpicture}
\end{center}
\caption{The $\{1,n\}$-grassmannian of $A_n$}
\label{Gr1n(An)}
\end{figure}\\
\noindent
The points of $\Gr_{{1,n}}(A_n(\KK))$ are flags of type $\{1,n\}$ in $A_n(\KK)$; its lines are flags of type either $\{2,n\}$ or
$\{1,n-1\}$; a point $p$ and a line $\ell$ are incident if and only if $p\cup\ell$ is a flag of $A_n(\KK)$.
Turning to $D_n$, we label the nodes of this diagram as in Fig.~\ref{Dn}.
\begin{figure}[h]
        \begin{center}
          \begin{tikzpicture}(50,7)
            \tikzstyle{every node}=[draw,circle,fill=black,minimum size=2mm,
            label distance=4pt, inner sep=2pt]
            \draw
            (0,0) node[label=above:$1$](1) {}
            (1.5,0) node[label=above:$2$](2) {}
            (3,0) node[label=above:$3$](3) {}
            (4.5,0) node[label={[yshift=-0.25cm]{$n-3$}}](n-3) {}
            (6,0) node[label={[yshift=-0.25cm]$n-2$}](n-2) {}
            (8,0.8) node[label=right:$ +$](+) {}
            (8,-0.8) node[label=right:$-$](-) {};
             \draw (1)--(2)--(3);
 \draw (n-3)--(n-2)--(+);
 \draw (n-2)--(-);
 \draw [dashed] (3)--(n-2);
\end{tikzpicture}
\end{center}
\caption{Labeling of types for buildings of type $D_{n}$}
\label{Dn}
\end{figure}
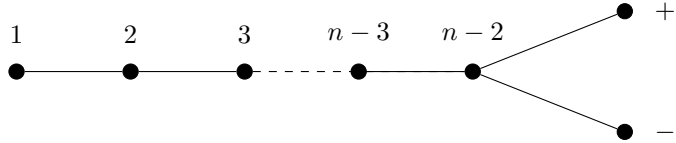\\
\noindent
We are interested in the $J$-grassmannians  $\Gr_{J}(D_n(\KK))$, where
$J=\{+,-\}$ or $J=\{1,+,-\}$ or $J=\{1,-\}$ (we can omit the case $J = \{1,+\}$ since $\Gr_{1,+}(D_n(\KK)) \cong \Gr_{1,-}(D_n(\KK))$); see Fig.~\ref{DnGeom}.
\begin{figure}[t]
        \begin{center}
          \begin{tikzpicture}(50,7)
            \tikzstyle{every node}=[draw,circle,fill=black,minimum size=2mm,
            label distance=4pt, inner sep=2pt]
            \node [draw=none,rectangle,fill=none] at (-2,0) {$\Gr_{{+,-}}(D_n):$};
            \draw
            (0,0) node[label=above:$1$](1) {}
            (1.5,0) node[label=above:$2$](2) {}
            (3,0) node[label=above:$3$](3) {}
            (4.5,0) node[label={[yshift=-0.25cm]{$n-3$}}](n-3) {}
            (6,0) node[label={[yshift=-0.25cm]$n-2$}](n-2) {}
            (8,0.8) node[label=right:$ +$](+) {}
            (8,-0.8) node[label=right:$-$](-) {};
             \draw (1)--(2)--(3);
 \draw (n-3)--(n-2)--(+);
 \draw (n-2)--(-);
 \draw [dashed] (3)--(n-2);
 \draw (+) circle (6pt);
 \draw (-) circle (6pt);
 \end{tikzpicture}
 \vskip.9cm
 \begin{tikzpicture}(50,7)
            \tikzstyle{every node}=[draw,circle,fill=black,minimum size=2mm,
            label distance=4pt, inner sep=2pt]
              \node [draw=none,rectangle,fill=none] at (-2,0) {$\Gr_{{1,+,-}}(D_n):$};
            \draw
            (0,0) node[label=above:$1$](1) {}
            (1.5,0) node[label=above:$2$](2) {}
            (3,0) node[label=above:$3$](3) {}
            (4.5,0) node[label={[yshift=-0.25cm]{$n-3$}}](n-3) {}
            (6,0) node[label={[yshift=-0.25cm]$n-2$}](n-2) {}
            (8,0.8) node[label=right:$ +$](+) {}
            (8,-0.8) node[label=right:$-$](-) {};
             \draw (1)--(2)--(3);
 \draw (n-3)--(n-2)--(+);
 \draw (n-2)--(-);
 \draw [dashed] (3)--(n-2);
 \draw (+) circle (6pt);
 \draw (-) circle (6pt);
 \draw (1) circle (6pt);
\end{tikzpicture}
\vskip.9cm
\begin{tikzpicture}(50,7)
            \tikzstyle{every node}=[draw,circle,fill=black,minimum size=2mm,
            label distance=4pt, inner sep=2pt]
            \node [draw=none,rectangle,fill=none] at (-2,0) {$\Gr_{{1,-}}(D_n):$};
            \draw
            (0,0) node[label=above:$1$](1) {}
            (1.5,0) node[label=above:$2$](2) {}
            (3,0) node[label=above:$3$](3) {}
            (4.5,0) node[label={[yshift=-0.25cm]{$n-3$}}](n-3) {}
            (6,0) node[label={[yshift=-0.25cm]$n-2$}](n-2) {}
            (8,0.8) node[label=right:$ +$](+) {}
            (8,-0.8) node[label=right:$-$](-) {};
             \draw (1)--(2)--(3);
 \draw (n-3)--(n-2)--(+);
 \draw (n-2)--(-);
 \draw [dashed] (3)--(n-2);
\draw (-) circle (6pt);
 \draw (1) circle (6pt);
\end{tikzpicture}
\end{center}
\caption{Geometries associated to buildings of type $D_n$}
\label{DnGeom}
\end{figure}
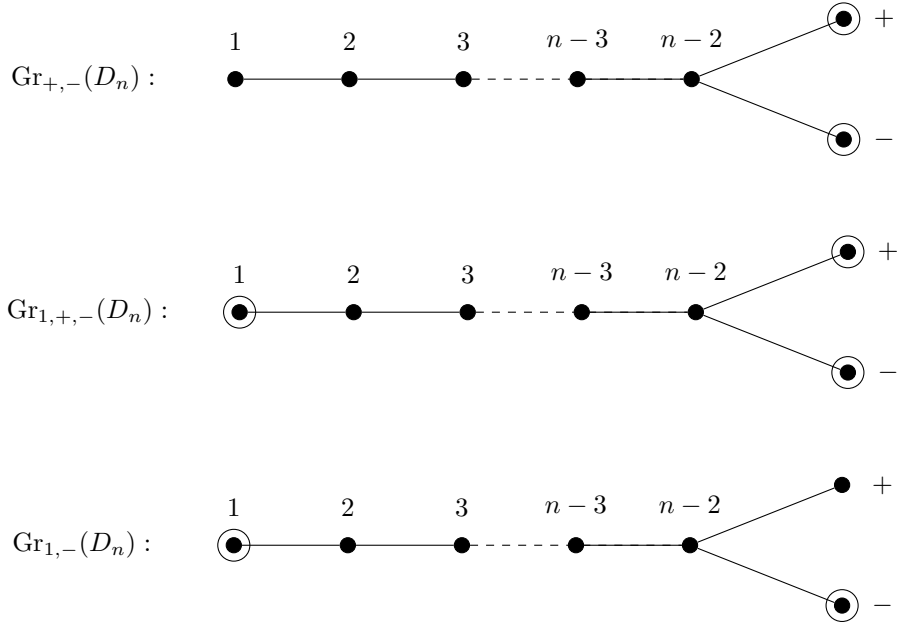
Explicitly, the points of $\Gr_{{+,-}}(D_n(\KK))$ are the flags of $D_n(\KK)$ of type $\{+,-\}$ while the lines
are the flags of types  $\{n-2,+\}$ and $\{n-2,-\}$ with incidence between a point $p$ and a line $\ell$ given
by the condition that $p\cup\ell$ must be a flag of $D_n(\KK)$. As for $\Gr_{{1,+,-}}(D_n(\KK))$, its points are the flags of type $\{1,+,-\}$,
and the lines are the flags of type $\{2,+,-\}$, $\{1,n-2,+\}$ or
$\{1,n-2,-\}$; incidence is defined as above. Finally, the points of $\Gr_{{1,-}}(D_n(\KK))$ are the flags of type $\{1,-\}$
and the lines are the flags of type either $\{2,-\}$ or $\{1,n-2\}$.
\noindent
Note that when $n = 3$, since $D_3(\KK)\cong A_3(\KK)$, we have $\Gr_{{1,3}}(A_3(\KK)) \cong \Gr_{{+,-}}(D_3(\KK))$. In any case, $\Gr_{{+,-}}(D_n(\KK)) \cong \Gr_{n-1}(B_n^+(\KK))$, where $B_n^+(\KK):=\Gr_{1}(D_n(\KK))$ is the $1$-grassmannian of $D_n(\KK)$ (but regarded as a geometry of rank $n$), namely the top-thin polar space associated to the group $\mathrm{O}^+(2n, \KK)$; see Fig.~\ref{BnGeom}.
\par
In the following we shall add more details on the grassmannians introduced before.
In particular, we shall better explain in which sense $\Gr_J(A_n(\KK))$ and $\Gr_J(D_n(\KK))$ contain $\Gr_J(A_n(\KK_0))$ and $\Gr_J(D_n(\KK_0))$ for a sub-division ring $\KK_0$ of $\KK$.
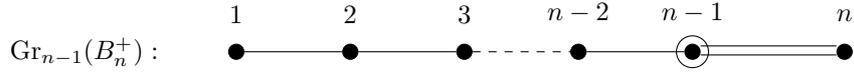
\begin{figure}[b]
        \begin{center}
          \begin{tikzpicture}(50,7)
            \tikzstyle{every node}=[draw,circle,fill=black,minimum size=2mm,
            label distance=4pt, inner sep=2pt]
            \node [draw=none,rectangle,fill=none] at (-2,0) {$\Gr_{n-1}(B_n^+):$};
            \draw
            (0,0) node[label=above:$1$](1) {}
            (1.5,0) node[label=above:$2$](2) {}
            (3,0) node[label=above:$3$](3) {}
            (4.5,0) node[label={[yshift=-0.25cm]{$n-2$}}](n-2) {}
            (6,0) node[label={[yshift=-0.25cm]$n-1$}](n-1) {}
            (8,0) node[label={[yshift=0.0cm]:$n$}](n) {};
            \draw (1)--(2)--(3);
            \draw (n-2)--(n-1);
            \draw [transform canvas={yshift=+2.2pt}] (n-1)--(n);
            \draw [transform canvas={yshift=-1.2pt}] (n-1)--(n);
            \draw [dashed] (3)--(n-2);
            \draw (n-1) circle (6pt);
          \end{tikzpicture}
        \end{center}
        \caption{Geometry $\Gr_{n-1}(B_n^+(\KK)) \cong \Gr_{+,-}(D_n(\KK))$}
        \label{BnGeom}
      \end{figure}
\clearpage

\subsection{The geometry $A_n(\KK)$ and its grassmannian $\Gr_{1,n}(A(\KK))$}\label{models1}

Let $A_n(\KK)$ be a geometry of type $A_n$ defined over a
division ring $\KK$, with $n\geq 3$. Explicitly, $A_n(\KK)\cong\PG(V_{n+1}(\KK))$ for a  $(n+1)$-dimensional right
$\KK$-vector space $V_{n+1}(\KK)$. For $i=1,2,\dots,n$ the elements of $A_n(\KK)$ of type
$i$ are the $i$-dimensional subspaces of $V_{n+1}(\KK)$, with symmetrized inclusion as the incidence relation.
As customary we call the elements of $A_n(\KK)$ of type $1$, $2$ and $n$ {\em points}, {\em lines} and {\em hyperplanes} respectively. The elements of type $n-1$ will be called {\em sub-hyperplanes}. Note that, when $n = 3$, lines and sub-hyperplanes are the same objects.

Turning to $\Gr_{{1,n}}(A_n(\KK))$, its points are the point-hyperplane flags $(p,H)$ of $A_n(\KK)$. Its lines, regarded as sets of points, are of either of the following two types:
\begin{align}
\label{Anlines:a} &\ell_{p,S}:=\{ (p,X) \colon X \mbox{ hyperplane, } X\supset S\}\mbox{ for a (point, sub-hyperplane) flag } (p,S). \\
\label{Anlines:b} &\ell_{L,H}:=\{ (x,H) \colon x \mbox{ a point, } x\subset L\}, \mbox{ for a line-hyperplane flag } (L,H).
\end{align}

\subsection{$D_n(\KK)$ and $\Gr_J(D_n(\KK))$ for $J = \{+,-\}, \{1,-\}$ or $\{1,+,-\}$}\label{models2}

Let $\KK$ be a field and  $V_{2n}(\KK)$ a vector space of dimension $2n$ over $\KK$, with $n\geq 3$. Consider
a non-degenerate quadratic form $q$ on $V_{2n}(\KK)$ of Witt index $n$. As in Section~\ref{mres}, let $B_n^+(\KK)$ be the polar space associated to $q$, namely the (weak) building of rank $n$ whose elements are the vector subspaces of $V_{2n}(\KK)$ that are totally singular with respect to $q$, with their dimensions taken as types. The elements of $B^+_n(\KK)$ of dimension $1$ are called {\em points} and those of dimension $2$ \emph{lines}.

It is well known that we can ``unfold'' $B^+_n(\KK)$ as to obtain a building $D_n(\KK)$ of type $D_n$ (see e.g. Tits~\cite[Chapter 7]{Tits}). Explicitly, let $\sim$ be the equivalence relation on the set of all $n$-dimensional subspaces of $B_n^+(\KK)$ defined as follows: $X\sim Y$ if and only if $X\cap Y$ has even codimension in $X$ (equivalently, in $Y$).  Let $\fS^+$ and $\fS^-$ be the two equivalence classes of $\sim$. Take $\{1,2,\dots,n-2,+,-\}$ as the set of types. For $1\leq i\leq n-2$ the $i$-elements of $B_n^+(\KK)$ are the elements of $D_n(\KK)$ of type $i$ and the elements of $\fS^+$ and $\fS^-$ are given types $+$ and $-$ respectively. The $(n-1)$-elements of $B_n^+(\KK)$ are dropped (but we can recover them as flags of type $\{+,-\}$). Incidence between elements of different types $\{i,j\}$ with $\{i,j\}\neq\{+,-\}$ is symmetrized inclusion; if $X\in\fS^+$ and $Y\in\fS^-$ then $X$ is incident with $Y$ if and only if $\dim(X\cap Y)=n-1$.

It is clear from the way  $D_n(\KK)$ is defined that the $1$-grassmannian $\Gr_1(D_n(\KK))$ of $D_n(\KK)$, regarded as a geometry of rank $n$, is just the same as $B^+_n(\KK)$. So, we can go back and forth from $D_n(\KK)$ to $B^+_n(\KK)$ as if they were the same object. In the sequel we will sometimes avail ourselves of this opportunity, when profitable.

Turning to grassmannians, $\Gr_{{+,-}}(D_n(\KK))$ is the point-line geometry whose points are the flags $(M_1,M_2)$ of $D_n(\KK)$ of type $(+,-)$ and whose
lines are of the following two forms:
\begin{align}
 \label{+-lines:a} \ell_{U,M_1}  :=  \{(M_1,X)\colon {X\in \fS^-}, ~ M_1\cap X \supset U \}
 \mbox{ with } (U, M_1) \mbox{ a flag of type } (n-2, +);\\
\label{+-lines:b} \ell_{U,M_2} :=  \{(X,M_2)\colon {X\in \fS^+}, ~ X\cap M_2 \supset U \}
 \mbox{ with } (U, M_2) \mbox{ a flag of type } (n-2, -).
\end{align}
Recall that the points of the grassmannian $\Gr_{n-1}(B_n^+(\KK))$ of $B^+_n(\KK)$ are the $(n-1)$-dimensional subspaces of $V_{2n}(\KK)$ totally singular for the quadratic form $q$ and the lines are the sets of the form $\ell_{X,M}  :=\{ Y : X\subset Y \subset M \}$ where $\dim(X)=n-2$, $\dim(M)=n$, $X\subset M$ and $M$ is totally singular. Every point $X$ of $\Gr_{n-1}(B_n^+(\KK))$ is the intersection $X = M_1\cap M_2$ of a unique pair $\{M_1,M_2\}$ of $n$-dimensional totally singular subspaces, which necessarily form a  $(+,-)$--flag of $D_n(\KK)$. Conversely, for every $(+,-)$--flag $(M_1, M_2)$ of $D_n(\KK)$, the intersection $X = M_1\cap M_2$ is a point of $\Gr_{n-1}(B^+_n(\KK))$. A bijecive mapping $\iota$ is thus naturally defined from the set of points of $\Gr_{n-1}(B^+_n(\KK))$ onto the set of points of $\Gr_{+,-}(D_n(\KK))$. The mapping $\iota$ induces a bijection from  the set of lines of $\Gr_{n-1}(B^+_n(\KK))$ onto the set of lines of $\Gr_{+,-}(D_n(\KK))$. In fact, if $\ell_{X,M}$ is a line of $\Gr_{n-1}(B^+_n(\KK))$ then $\iota(\ell_{X,M})$ is the line of $\Gr_{+,-}(D_n(\KK))$ denoted by the very same symbol $\ell_{X,M}$ and it has either form~\eqref{+-lines:a} or~\eqref{+-lines:b} according to whether $M$ belongs to $\fS^+$ or $\fS^-$. To sum up, $\Gr_{n-1}(B_n^+(\KK)) \cong \Gr_{{+,-}}(D_n(\KK))$.

The grassmannian $\Gr_{{1,-}}(D_n(\KK))$ is the point-line geometry where the points are the flags $(p,M)$ of $D_n(\KK)$ of type $(1,-)$ and the lines are as follows:
\begin{align}
 \label{1-lines:a}  &\ell_{p,U} :=\{ (p,X)\colon X\in \fS^-, X \supset U \} \mbox{ with } (p,U) \mbox{ a flag of type } (1, n-2); \\
  \label{1-lines:b} &\ell_{L,M}:=\{ (x,M)\colon \dim(x) = 1, x\subset L\} \mbox{ with } (L,M) \mbox{ a flag of type } (2, -).
\end{align}
The grassmannian $\Gr_{{1,+,-}}(D_n(\KK))$ is the point-line geometry where the points are the flags $(p,M_1,M_2)$ of $D_n(\KK)$ of type $(1,+,-)$;  the lines are as follows:
\begin{align}
   \label{1+-lines:a} &\ell_{L,M_1,M_2} :=\{ (p,M_1,M_2)\colon \dim(p) = 1, p\subset L\}
 \mbox{ with } (L, M_1, M_2) \mbox{ a flag of type }  (2, +, -); \\
  \label{1+-lines:b} &\ell_{p,U,M_1}  :=\{ (p,M_1,X)\colon {{X\in \fS^-}}, X\supset U\}
 \mbox{ with } (p, U, M_1)  \mbox{ a flag  of type } (1, n-2, +);\\
  \label{1+-lines:c}  &\ell_{p,U,M_2} :=\{ (p,X,M_2)\colon {{X\in \fS^+}}, X \supset U\}
  \mbox{ with } (p, U, M_2) \mbox{ a flag of type } (1, n-2, -).
\end{align}

\subsection{The subgeometry $\Gr_J(A_n(\KK_0))$ of $\Gr_J(A_n(\KK))$ for $\KK_0 \leq \KK$}\label{rational1}

Let $E$ be a basis of $V_{n+1}(\KK)$ and $\KK_0$ a sub-division ring of $\KK$. We say that a vector $v\in V_{n+1}(\KK)$ is $\KK_0$-{\em rational respect to} $E$ (also $\KK_0$-{\em rational} for short, when the basis $E$ is clear from the context) if $v$ is a linear combination of vectors of $E$ with coefficients in $\KK_0$. The set of $\KK_0$-rational vectors (with respect to $E$) is a $\KK_0$-vector space, henceforth denoted $V_{n+1, E}(\KK_0)$. For a subspace $X$ of $V_{n+1}(\KK)$, let $X_0 := X\cap V_{n+1,E}(\KK_0)$. Clearly, $X_0$ is a subspace of $V_{n+1,E}(\KK_0)$. We say that $X$ is $\KK_0$-{\em rational with respect to} $E$ (also $\KK_0$-{\em rational} for short) if $X_0$ spans $X$ (in $V_{n+1}(\KK)$); in other words $X$, as a subspace of $V_{n+1}(\KK)$, admits a basis formed by $\KK_0$-rational vectors. If this is the case, then $X$ and $X_0$ have the same dimension (in $V_{n+1}(\KK)$ and $V_{n+1,E}(\KK_0)$ respectively). Indeed the rank of a matrix $M$ with entries in $\KK_0$ does not change if $M$ is regarded as matrix with entries in $\KK$).

Clearly, the sum of two $\KK_0$-rational subspaces of $V_{n+1}(\KK)$ is still $\KK_0$-rational. Similarly,
\begin{lemma}\label{rational An 2}
The intersection of two $\KK_0$-rational subspaces is still $\KK_0$-rational.
\end{lemma}
\begin{proof}
  Let $X_0, Y_0$ be two subspaces of $V_{n+1, E}(\KK_0)$ and $X, Y$ be their spans in $V_{n+1}(\KK)$. Then $X\cap Y$ contains the span $Z$ of $X_0\cap Y_0$ in $V_{n+1}(\KK)$. We must prove that $X\cap Y = Z$. Clearly, $\dim(X) = \dim(X_0)$ and $\dim(Y) = \dim(Y_0)$. Moreover $\dim(X+Y) = \dim(X_0+Y_0)$. Hence
  \begin{multline*}
    \dim(X\cap Y) = \dim(X)+\dim(Y)- \dim(X+Y) = \\  \dim(X_0) + \dim(Y_0)-\dim(X_0+Y_0) = \dim(X_0\cap Y_0) = \dim(Z).
  \end{multline*}
    Therefore $X\cap Y = Z$.
\end{proof}

The following is now obvious:
\begin{proposition}\label{rational An 3}
The $\KK_0$-rational elements of $A_n(\KK)$ form a geometry $A_{n,E}(\KK_0) \cong A_n(\KK_0)$.
\end{proposition}

In view of Proposition~\ref{rational An 3}, we can freely identify $A_n(\KK_0)$ with $A_{n, E}(\KK_0)$, thus regarding $A_n(\KK_0)$ as a subgeometry of $A_n(\KK)$. The flags of $A_n(\KK_0)$ are thus identified with the $\KK_0$-{\em rational} flags of $A_n(\KK)$, namely the flags of $A_n(\KK)$ all elements of which are $\KK_0$-rational (with respect to the selected basis $E$ of $V_{n+1}(\KK)$). Accordingly, for $\emptyset \neq J\subseteq \{1,2,\dots, n\}$ the $J$-grassmannian $\Gr_J(A_n(\KK_0))$ of $A_n(\KK_0)$ is identified with the subgeometry $\Gr_{J, E}(A_n(\KK_0))$ of $\Gr_J(A_n(\KK))$ formed by the $\KK_0$-{\em rational} points and lines of $\Gr_J(A_n(\KK))$, namely the points and lines of $\Gr_J(A_n(\KK))$ which are $\KK_0$-rational as flags of $A_n(\KK)$.

Henceforth, by a harmless little abuse, we will always regard $\Gr_J(A_n(\KK_0))$ as the same as $\Gr_{J,E}(A_n(\KK_0))$, thus referring to the span of $\Gr_J(A_n(\KK_0))$ in $\Gr_J(A_n(\KK))$, as we have done in the Introduction, while in fact we mean the span of $\Gr_{J,E}(A_n(\KK_0))$.

The next proposition states that, regarding $\Gr_J(A_n(\KK_0))$ as a subgeometry of $\Gr_J(A_n(\KK))$, the collinearity graph of $\Gr_J(A_n(\KK_0))$ is just the graph induced on its point-set by the collinearity graph of $\Gr_J(A_n(\KK))$.
\begin{proposition}\label{rational An 4}
A line of $\Gr_J(A_n(\KK))$ is $\KK_0$-rational if and only if at least two of its points are $\KK_0$-rational.
\end{proposition}
\begin{proof}
The `only if' part of this claim easily follows from the isomorphism of the geometries $\Gr_{J,E}(A_n(\KK_0))\cong \Gr_J(A_n(\KK_0))$. Turning to the `if' part, given $j_0\in  J$, let $L$ be a flag of $A_n(\KK)$ of type $(J\setminus\{j_0\})\cup\mathrm{fr}(j_0)$ and let $P$ and $P'$ be two distinct $J$-flags of $A_n(\KK)$ incident with $L$. We must prove that if both $P$ and $P'$ are $\KK_0$-rational then $L$ is also $\KK_0$-rational. There are three cases to examine: $J$ contains elements $j < j_0$ as well elements $j' > j_0$; $j_0 \leq j$ for every $j\in J$; $j_0\geq j$ for every $j\in J$. We shall examine only the first case, leaving the remaining two (easier) cases to the reader.

With $j_0$ as in the first case, the flag $L$ has type $(J\setminus\{j_0\})\cup\{j_0-1, j_0+1\}$ and contains $Q := P\cap P'$, which is a flag of type $J\setminus \{j_0\}$. Moreover, there are distinct $j_0$-subspaces $S, S'$ of $V_{n+1}(\KK)$ incident with $L$ such that $P = Q\cup\{S\}$ and $P' = Q\cup\{S'\}$. As $S$ and $S'$ are incident with $L$, the elements of $L$ of type $j_0-1$ and $j_0+1$ coincide with $S\cap S'$ and $S+S'$ respectively, namely $L = Q\cup\{S\cap S', S+S'\}$. By assumption, $P$ and $P'$ are $\KK_0$-rational. Hence $Q = P\cap P'$ as well as $S$ and $S'$ are $\KK_0$-rational. If $j_0-1\in J$ then $S\cap S' \in Q$, hence $S\cap S'$ is $\KK_0$-rational. Otherwise $S\cap S'$ is $\KK_0$-rational by Lemma \ref{rational An 2}. Similarly, $S+S'$ is $\KK_0$-rational. Thus, all elements of $L$ are $\KK_0$-rational, namely $L$ is $\KK_0$-rational.
\end{proof}

\subsection{The subgeometry $\Gr_J(D_n(\KK_0))$ of $\Gr_J(D_n(\KK))$ for $\KK_0 \leq \KK$}\label{rational2}

Let $\KK_0$ be a subfield of $\KK$. Let $q:V_{2n}(\KK)\rightarrow \KK$ be the quadratic form considered in Section \ref{models2}. Without loss of generality we can assume to have chosen the basis $E = (e_1,\dots,e_{2n})$ of $V_{2n}(\KK)$ in such a way that $q$ admits the following canonical expression with respect to $E$:
\begin{equation}\label{forma quadratica}
q(x_1,\dots, x_{2n}) ~ = ~ x_1x_2+\dots +x_{2n-1}x_{2n}.
\end{equation}
As in Section~\ref{rational1}, we can consider the $\KK_0$-vector space $V_{2n,E}(\KK_0)$ formed by the $\KK_0$-rational vectors (with respect to $E$).
The form $q$ induces a quadratic form $q_0$ on $V_{2n, E}(\KK_0)$. Clearly, a $\KK_0$-rational subspace $X$ of $V_{2n}(\KK)$ is totally singular for $q$ if and only if $X\cap V_{2n,E}(\KK_0)$ is totally singular for $q_0$. Hence the polar space $B^+_n(\KK_0)$ associated to $q_0$ can be identified with the subgeometry $B^+_{n,E}(\KK_0)$ of $B^+_n(\KK)$ formed by the $\KK_0$-rational subspaces of $V_{2n}(\KK)$ which are totally singular for $q$. Similarly, $D_n(\KK_0)$ can be identified with the subgeometry $D_{n,E}(\KK_0)$ of $D_n(\KK)$ formed by the $\KK_0$-rational elements of $D_n(\KK)$.

A flag of $D_n(\KK)$ is $\KK_0$-{\em rational} if all of its elements are $\KK_0$-rational (with respect to $E$, of course).  Given a nonempty subset $J$ of the type-set $\{1, 2,\dots, n-2, +, -\}$ of $D_n(\KK)$, a point or a line of $\Gr_J(D_n(\KK))$ are said to be $\KK_0$-{\em rational} if they are $\KK_0$-rational as flags of $D_n(\KK)$. The $\KK_0$-rational points and lines of $\Gr_J(D_n(\KK))$ form a subgeometry $\Gr_{J,E}(D_n(\KK_0))$ of $\Gr_J(D_n(\KK))$ isomorphic to $\Gr_J(D_n(\KK_0))$. An analogue of Proposition \ref{rational An 4} also holds:

\begin{proposition}\label{rational Dn}
A line of $\Gr_J(D_n(\KK))$ is $\KK_0$-rational if and only if at least two of its points are $\KK_0$-rational.
\end{proposition}
\begin{proof}
This statement can be proved in the same way as Proposition \ref{rational An 4} but for a couple of cases in the proof of the `only if' part, which we shall now discuss.

\medskip

\noindent
1. Suppose that $J$ contains at least one of the types $+$ and $-$, say $+ \in J$. Suppose moreover that $n-2\not\in J$. Let $L$ be a flag of $D_n(\KK)$ of type $(J\setminus\{+\})\cup\mathrm{fr}(+) = (J\setminus\{+\})\cup\{n-2\}$ and let $P, P'$ be distinct $\KK_0$-rational flags of type $J$, both incident with $L$. Then $Q = P\cap P'$ is a $\KK_0$-rational flag, $P = Q\cup\{M\}$ and $P' = Q\cup\{M'\}$ for distinct $\KK_0$-rational element $M, M'\in \fS^+$. Also, $L = Q\cup S$ for an $(n-2)$-element $S$ incident with $Q$. We have $S\subseteq M\cap M'$ since $P$ and $P'$ are incident with $L$. However, $\mathrm{dim}(M\cap M')$ has even codimension in $M$ and $M'$, since $M$ and $M'$ belong to the same family of $n$-elements of $B^+_n(\KK)$, namely $\fS^+$. Therefore $S = M\cap M'$. Hence $S$ is $\KK_0$-rational by Lemma \ref{rational An 2}. Thus, $L$ is $\KK_0$-rational.

\medskip

\noindent
2. The set $J$ contains none of the types $+$ or $-$ but it contains $n-2$. To fix ideas, suppose that $n > 3$. Let $L$ be a flag of $D_n(\KK)$ of type $(J\setminus\{n-2\})\cup\mathrm{fr}(n-2) = (J\setminus\{n-2\})\cup\{n-3,+,-\}$ and let $P, P'$ be distinct $\KK_0$-rational flags of type $J$, both incident with $L$. Then $Q = P\cap P'$ is a $\KK_0$-rational flag, $P = Q\cup\{S\}$ and $P' = Q\cup\{S'\}$ for distinct $\KK_0$-rational $(n-2)$-elements $S, S'$ of $D_n(\KK)$ and $L = Q\cup\{R, M_1, M_2\}$ for an $(n-3,+,-)$--flag $(R, M_1, M_2)$ incident with $Q$. As both $P$ and $P'$ are incident with $L$, the sum $S+S'$ is contained in $M\cap M'$. However $\dim(M\cap M) = n-1$ while $\dim(S+S')\geq n-1$ since $S\neq S'$. Consequently, $M\cap M' = S+S'$. On the other hand, $S+S'$ is a $\KK_0$-rational subspace of $V_{2n}(\KK)$, since both $S$ and $S'$ are $\KK_0$-rational. Hence $M\cap M'$ is $\KK_0$-rational. Therefore $M\cap M'$ is an $(n-1)$-element of $B^+_{n,E}(\KK_0) = \Gr_1(D_{n,E}(\KK_0))$. Accordingly, $M\cap M' = M_0\cap M'_0$ for a $(+,-)$--flag $(M_0, M'_0)$ of $D_{n,E}(\KK_0)$. On the other hand, all $(+,-)$--flags of $D_{n,E}(\KK_0)$ are $(+,-)$--flags of $D_n(\KK)$ too and two $(+,-)$--flags $(M, M')$ and $(M_0, M'_0)$ of $D_n(\KK)$ coincide if $M\cap M' = M_0\cap M'_0$. It follows that $M = M_0$ and $M' = M_0$, namely both $M$ and $M'$ are $\KK_0$-rational. It remains to prove that $R$ too is $\KK_0$-rational. If $n-3 \in J$ then $R\in Q$ and there is nothing to prove. Otherwise $R = S\cap S'$. Hence $R$ is $\KK_0$-rational by Lemma \ref{rational An 2}. Therefore $L$ is $\KK_0$-rational.

We have assumed that $n > 3$. When $n = 3$ we have $J = \{n-2\}$ and $L = (M_1, M_2)$, of type $(+,-)$; we get the conclusion as above, but now with no $R$ to take care of.
\end{proof}

All we have said for $D_n(\KK_0)$ and $\Gr_J(D_n(\KK_0))$ in this section  holds for $B_n^+(\KK_0)$ and $\Gr_J(B^+_n(\KK_0))$ as well.

\section{Proof of Theorems \ref{m-th1} and \ref{non-interval}}\label{pro}

For $X_n$ equal to $A_n$ or $D_n$ and a nonempty set of types $J$, let $\Gamma(\KK):=\Gr_J(X_n(\KK))$ and $\Gamma(\KK_0) := \Gr_J(X_n(\KK_0))$ be its $\KK_0$-rational subgeometry for a proper sub-division ring $\KK_0$ of $\KK$ (Sections \ref{rational1} and \ref{rational2}).

\begin{definition}\label{splits}
We say that a node $t$ of $X_n$ {\em splits} $J$ if $t\not \in J$ and $J$ is not contained in one single connected component of $X_n\setminus \{t\}$. In other words, $t$ separates at least two of the types of $J$.
\end{definition}

\begin{definition}\label{Omega}
We say that a $J$-flag $F$ (point of $\Gamma(\KK)$) is {\em nearly $\KK_0$-rational} if either at least one of its elements is $\KK_0$-rational or there exists a $\KK_0$-rational element of $X_n(\KK)$ incident with $F$ and such that its type splits $J$. We denote by $\Omega_{\KK_0}(\Gamma(\KK))$ the set of all nearly $\KK_0$-rational points of $\Gamma(\KK)$.
\end{definition}
Obviously, $\Gamma(\KK_0) \subseteq \Omega_{\KK_0}(\Gamma(\KK))$. We shall prove the following:
\begin{theorem}\label{m-th3}
If $\Gamma(\KK)$ is $\Gr_{{1,n}}(A_n(\KK))$, $\Gr_{{1,-}}(D_n(\KK))$, $\Gr_{{1,+,-}}(D_n(\KK))$ or $\Gr_{+,-}(D_n(\KK))$ then $\Omega_{\KK_0}(\Gamma(\KK))$ is a proper subspace of $\Gamma(\KK)$.
\end{theorem}
Theorem \ref{m-th1} then immediately follows from Theorem \ref{m-th3} and the inclusion $\Gamma(\KK_0) \subseteq \Omega_{\KK_0}(\Gamma(\KK))$.

\subsection{Proof of Theorem \ref{m-th3}}

We need a preliminary result from multi-linear algebra, to be exploited later, when discussing the case $\Gamma(\KK) = \Gr_{+,-}(D_n(\KK))$.

\begin{lemma}\label{exterior}
Suppose that $\KK$ is a field and let $V := V_4(\KK)$. Given a basis $E = (e_1, e_2, e_3, e_4)$ of $V$, let $E\wedge E = (e_i\wedge e_j)_{i < j}$ be the corresponding basis of the second exterior power $V\wedge V$ of $V$. Then all the following hold:
\begin{enumerate}[(1)]
\item\label{ee:1} The span $\langle v, w\rangle$ of two independent vectors $v, w \in V$ is $\KK_0$-rational with respect to $E$ if and only if $v\wedge w$ is proportional to a vector of $V\wedge V$ which is $\KK_0$-rational with respect to $E\wedge E$.
\item\label{ee:2} A non-zero vector $v\in V$ is proportional to a $\KK_0$-rational vector if and only if the subspace $S_v := \langle v\wedge x\rangle_{x\in V}$ of $V\wedge V$ is $\KK_0$-rational with respect to $E\wedge E$.
\item\label{ee:3} The span $\langle u, v, w\rangle$ of three independent vectors $u, v, w\in V$ is $\KK_0$-rational (with respect to $E$) if and only if $\langle u\wedge v, u\wedge w, v\wedge w\rangle$ is $\KK_0$-rational with respect to $E\wedge E$.
\end{enumerate}
\end{lemma}
\begin{proof}
\begin{enumerate}[(1)]
\item Without loss of generality, we can assume  $v = e_1 + e_3a_3+e_4a_4$ and $w = e_2 + e_3b_3+ e_4b_4$ for $a_3, a_4, b_3, b_4 \in \KK$. Hence
$v\wedge w = e_{1,2} + e_{1,3}b_3 + e_{1,4}b_4 - e_{2,3}a_3 - e_{2,4}a_4 + e_{3,4}(a_3b_4-a_4b_3)$, where we write $e_{i,j}$ for $e_i\wedge e_j$.
Both parts of (\ref{ee:1}) are equivalent to the following single claim: $a_3, a_4, b_3, b_4 \in \KK_0$. Hence they are mutually equivalent.
\item
  Without loss of generality, we can assume that $v = e_1 + e_2a_2+e_3a_3 + e_4a_4$. Hence $S_v := \langle v\wedge e_2, v\wedge e_3,v\wedge e_4\rangle$. We have
\[v\wedge e_2 = e_{1,2} - e_{2,3}a_3 - e_{2,4}a_4, \hspace{2 mm} v\wedge e_3 = e_{1,3} + e_{2,3}a_2 - e_{3,4}a_4, \hspace{2 mm} v\wedge e_4 = e_{1, 4} + e_{2,4}a_2 + e_{3,4}a_3,\]
with $e_{i,j} := e_i\wedge e_j$, as above. Both parts of (\ref{ee:2}) are thus equivalent to this: $a_2, a_3, a_4\in \KK_0$. So Claim~(\ref{ee:2}) is proved.
\item
  Without loss of generality, we can assume that $u = e_1+e_4a$, $v = e_2+e_4b$ and $w = e_3+e_4c$. Hence
$u\wedge v = e_{1,2} + e_{1,4}b - e_{2,4}a$, $u\wedge w = e_{1,3} - e_{1,4}c - e_{3,4}a$ and $v\wedge w = e_{2,3} + e_{2,4}c - e_{3,4}b$.
Both parts of~(\ref{ee:3}) are equivalent to this: $a, b, c\in \KK_0$. Claim~(\ref{ee:3}) follows.
\end{enumerate}
\end{proof}

\begin{lemma}\label{Omega-geo}
If $\Gamma(\KK)$ is as in the hypotheses of Theorem {\rm \ref{m-th3}} then the set $\Omega_{\KK_0}(\Gamma(\KK))$ is a subspace of $\Gamma(\KK)$.
 \end{lemma}
  \begin{proof}
We must show that, for any two nearly $\KK_0$-rational collinear points $F, F'$ of $\Gamma(\KK)$, the line $\langle F, F'\rangle_{\Gamma(\KK)}$ is fully contained in $\Omega_{\KK_0}(\Gamma(\KK))$.
There are several cases to consider:
\begin{enumerate}[1)]
\item\label{omega-geo:1}  $\Gamma(\KK)=\Gr_{1,n}(A_n(\KK))$. Let $F =(p,H)$ and $F'=(p',H')$ be two distinct collinear points of $\Gamma(\KK)$, namely two point-hyperplane flags with either $p \neq p'$ and $H = H'$ or $p = p'$ but $H \neq H'$. Suppose moreover that $F$ and $F'$ are nearly $\KK_0$-rational.
\begin{enumerate}[(a)]
\item\label{1:a} Let $p = p'$ and $H\neq H'$. Now $\langle F, F'\rangle_{\Gamma(\KK)} = \ell_{p, S}=\{(p,X)\colon X \supset S, \dim(X)=n\}$
 where $S= H\cap H' \supset p$ is a sub-hyperplane containing $p$. By assumption, there exist $\KK_0$-rational subspaces $U_0$, $U_0'$ of $V_{n+1}(\KK)$ such that $p\subseteq U_0\subseteq  H$ and $p\subseteq U_0'\subseteq  H'$. The subspace $U_0\cap U_0'$ is $\KK_0$-rational by Lemma \ref{rational An 2}, it contains $p$ and is contained in $S$. Hence it is contained in every hyperplane $X \supset S$. As $U_0\cap U_0'$ is $\KK_0$-rational, the flag $(p,X)$ is nearly $\KK_0$-rational for every hyperplane $X\subset S$, namely $\ell_{p,S} \subseteq \Omega_{\KK_0}(\Gamma(\KK))$.
\item\label{1:b} Let $H = H'$ but $p \neq p'$. Then $\langle F, F'\rangle_{\Gamma(\KK)}=\ell_{L, H}=\{(x, H)\colon x\subset L, \dim(x)=1\}$ where $L = p+p' \subset H$ is the span of $p\cup p'$ in $V_{n+1}(\KK)$. The argument used in case (a) above can be dualized as follows. By assumption, there exist $\KK_0$-rational subspaces $U_0, U'_0$ of $V_{n+1}(\KK)$ such that $p\subseteq U_0\subseteq H$ and $p'\subseteq U'_0 \subseteq H'$. Clearly, $L \subseteq U_0+U'_0\subseteq H$. Hence $x\subseteq U_0+U'_0 \subseteq H$ for every $1$-subspace $x$ of $L$. However $U_0+U'_0$ is $\KK_0$-rational. Therefore $(x,H)$ is nearly $\KK_0$-rational. It follows that $\ell_{L,H}\subseteq \Omega_{\KK_0}(\Gamma(\KK))$.
\end{enumerate}

\item\label{omega-geo:2} $\Gamma(\KK)=\Gr_{{1,-}}(D_n(\KK))$. Let $F =(p,M)$ and $F'=(p',M')$ be two collinear points of $\Gr_{{1,-}}(D_n(\KK))$. Since $F$ and  $F'$ are collinear, either $p = p'$ or $M = M'$. The line $\langle F, F'\rangle_{\Gamma(\KK)}$ is as in \eqref{1-lines:a} or \eqref{1-lines:b}  according to whether $p = p'$ or $M = M'$. When $p = p'$ then  the same argument as in (a) of \ref{omega-geo:1}) does the job, with the only change that $M\cap M'$, which now plays the role of $H\cap H'$, has dimension $n-2$ instead of $n-1$. If $M = M'$ then an argument similar to that used for (b) of
  \ref{omega-geo:1}) yields the conclusion. We leave the details to the reader.

\item\label{omega-geo:3} $\Gamma(\KK)=\Gr_{1,+,-}(D_n(\KK))$. Let $F = (p, M_1, M_2)$ and $F' = (p', M'_1, M_2')$ be two collinear points of $\Gamma(\KK)$ and suppose they both are nearly $\KK_0$-rational. Two subcases can occur:
\begin{enumerate}[(a)]
\item $M_i = M'_i$ for $i = 1, 2$. If at least one of the $n$-spaces $M_1$ and $M_2$ is $\KK_0$-rational, there is nothing to prove. Suppose that neither of them is $\KK_0$-rational. Then, since $F$ and $F'$ are nearly $\KK_0$-rational by assumption, there are $\KK_0$-rational subspaces $U_0$ and $U_0'$ with $p\subseteq U_0\subset M_1\cap M_2$ and $p'\subseteq U_0'\subset M_1\cap M_2$.  We have $\langle F, F'\rangle_{\Gamma(\KK)} = \ell_{L,M_1, M_2}$ as in  \eqref{1+-lines:a} with $L = p+p'$. The sum $U_0+U'_0$ is a $\KK_0$-rational subspace of $V_{2n}(\KK)$ and contains $L$.

  If $\dim(U_0+U'_0) < n-1$ then $U_0+U'_0$ is a $\KK_0$-rational element of $D_n(\KK)$ incident with the flag $(L, M_1, M_2)$, which corresponds to the line $\ell_{L,M_1,M_2}$. As in (b) of \ref{omega-geo:1}), it follows that all points of $\ell_{L,M_1,M_2}$ are nearly $\KK_0$-rational.

 If $\dim(U_0+U'_0) > n-2$ then necessarily $U_0+U'_0 = M_1\cap M_2$. In this case $U_0+U_0'$ is not an element of $D_n(\KK)$, but it is a $\KK_0$-rational $(n-1)$-element of $B^+_n(\KK)$, hence an $(n-1)$-element of the subgeometry $B^+_n(\KK_0)$ of $B^+_n(\KK)$. As such, $U_0+U'_0$ is contained in just two $n$-elements $N_1$ and $N_2$ of $B^+_n(\KK_0)$. However $N_1$ and $N_2$ also belong to $B^+_n(\KK)$. In fact, they are the unique two $n$-elements of $B^+_n(\KK)$ which contain $U_0+U'_0$. On the other hand, $U_0+U'_0$ is contained in $M_1$ and $M_2$. Therefore $\{M_1, M_2\} = \{N_1, N_2\}$. However $N_1$ and $N_2$ are $\KK_0$-rational. Hence $M_1$ and $M_2$ are $\KK_0$-rational, contrary to our assumptions. We have reached a contradiction. The proof is complete, as far as the present subcase is concerned.

\item Let $p=p'$, $M_i = M'_i$ but $M_j \neq M'_j$, for $\{i,j\} = \{1, 2\}$. To fix ideas, assume that $M_1 = M'_1$ and $M_2\neq M'_2$. If $M_1$ or $p$ are $\KK_0$-rational, then there is nothing to prove. Suppose that neither $M_1$ nor $p$ are $\KK_0$-rational. Recalling that $F$ and $F'$ are nearly $\KK_0$-rational, one of the following occurs:
\begin{enumerate}[(b1)]
\item\label{3:b1} There are $\KK_0$-rational subspaces $U_0$, $U_0'$ of dimension at most $n-2$ such that $p\subseteq U_0\subset M_1\cap M_2$ and $p\subseteq U_0'\subset M_1\cap M'_2$.
\item\label{3:b2} Just one of $M_2$ and $M'_2$ is $\KK_0$-rational. To fix ideas, let $M'_2$ be the $\KK_0$-rational one. Then there exists a $\KK_0$-rational subspace $U_0$ of dimension $\dim(U_0) \leq n-2$ such that $p\subseteq U_0\subset M_1\cap M_2$.
\item\label{3:b3} Both $M_2$ and $M'_2$ are $\KK_0$-rational.
\end{enumerate}

In subcases (b1) and (b2) we can consider the element $U_0\cap U'_0$ or $U_0\cap M'_2$ respectively. This element contains $p$ and is $\KK_0$-rational by Lemma \ref{rational An 2}. So, we get the conclusion as in (a) of \ref{omega-geo:1}). In subcase (b3), the intersection $U_0 = M_2\cap M'_2$ is
$\KK_0$-rational by Lemma~\ref{rational An 2} and has dimension $\dim(U_0) = n-2k$ for a positive integer $k < n/2$, since $M_2$ and $M_2'$ belong to the same class $\fS^-$. Hence $\dim(U_0) \leq n-2$. Moreover, $U_0\subset M_1$, since both $M_2$ and $M_2'$ are incident with $M_1$ in $D_n(\KK)$. Clearly, $p \subseteq  U_0$. Again, the conclusion follows as in (a) of \ref{omega-geo:1}).
\end{enumerate}
\item\label{omega-geo:4} $\Gamma(\KK) = \Gr_{+,-}(D_n(\KK))$. Assume firstly that $n = 3$. We have discussed this case in \cite[Theorem 5.10]{ILP19} but we turn back to it here, using an argument different from that of \cite{ILP19}.

By the Klein correspondence,  $V_{2n}(\KK) = V_6(\KK)$ can be regarded as the exterior square of $V_4(\KK)$, with the basis $E = (e_1,\dots, e_6)$ of $V_6(\KK)$, to be chosen as in Section~\ref{rational2}, realized as the exterior square $E = E'\wedge E'$ of a suitable basis $E'$ of $V_4(\KK)$. The elements of $D_3(\KK)$ of type $+$ or $-$ correspond to $1$- and $3$-dimensional subspaces of $V_4(\KK)$ and the $1$-elements of $D_3(\KK)$ correspond to $2$-subspaces of $V_4(\KK)$. By Lemma~\ref{exterior}, an element of $D_3(\KK)$ is $\KK_0$-rational with respect to $E$ if and only if the subspace corresponding to it in $V_4(\KK)$ is $\KK_0$-rational with respect to $E'$. Accordingly, a $(+,-)$--flag of $D_3(\KK)$ is nearly $\KK_0$-rational if and only if the corresponding $(1,3)$-flag of $A_3(\KK)$ is nearly $\KK_0$-rational. Thus, we are driven back to the special case $\Gr_{1,3}(A_3(\KK))$ of $\Gr_{1,n}(A_n(\KK))$, already discussed in \ref{omega-geo:1}) of this proof. It follows that $\Omega_{\KK_0}(\Gamma(\KK))$ is a subspace of $\Gamma(\KK)$, as claimed.

Consider now $n > 3$. Let $F = (M_1,M_2)$ and $F' = (M'_1, M'_2)$ be two distinct nearly $\KK_0$-rational collinear points of $\Gamma(\KK)$. As $F$ and $F'$ are collinear, either $M_1 = M'_1$ or $M_2 = M'_2$. To fix ideas, let $M_2 = M_2'$. Hence
\[\ell_{U,M_2} ~ = ~ \{(M,M_2) \colon M\in\fS^+, M\cap M_2 \supset U\}\]
is the line of $\Gamma(\KK)$ through $F$ and $F'$, where $U = M_1\cap M'_1\subset M_2$, $\dim(U) = n-2$ (see \eqref{+-lines:b}). If $M_2$ is $\KK_0$-rational, there is nothing to prove. Assuming that $M_2$ is not $\KK_0$-rational, there are still a number of subcases to examine:
\begin{enumerate}[(a)]
\item\label{4:a} Both $M_1$ and $M'_1$ are $\KK_0$-rational. Hence $U = M_1\cap M_1'$ is $\KK_0$-rational. Accordingly, every $(+,-)$--flag $(M,M_2)\in \ell_{U,M_2}$ is nearly $\KK_0$-rational.
\item\label{4:b} Neither $M_1$ nor $M'_1$ are $\KK_0$-rational. Hence there exist $\KK_0$-rational $(n-2)$-elements $U_0$ and $U_0'$ such that
  $U_0 \subset M_1\cap M_2$ and $U'_0 \subset M'_1\cap M_2$.

  If $U_0 = U'_0$ then $U_0 = M_1\cap M'_1$. However $M_1\cap M'_1 = U$.
Hence $U = U_0$ is $\KK_0$-rational. In this case we are done: all $(+,-)$--flags incident to $U$ are nearly $\KK_0$-rational.

On the other hand, suppose $U_0 \neq U'_0$. Then $U_0+U'_0$ is a $\KK_0$-rational $(n-1)$-dimensional subspace of $M_2$. Being $\KK_0$-rational, $U_0+U'_0$ is an $(n-1)$-element of $B^+_n(\KK_0)$. As such, it is contained in just two $n$-elements of $B^+_n(\KK_0)$. In other words, both
$n$-elements of $B^+_n(\KK)$ containing $U_0+U'_0$ are $\KK_0$-rational. However $M_2$ is indeed one of those two elements. Therefore $M_2$ is $\KK_0$-rational. This contradicts the assumptions made on $M_2$. Consequently, this case we have now been considering cannot occur.

\item\label{4:c}
  Just one of $M_1$ and $M'_1$ is $\KK_0$-rational. To fix ideas, let $M_1$ be the $\KK_0$-rational one. As $(M'_1, M_2)$ is nearly $\KK_0$-rational by assumption, but neither $M'_1$ nor $M_2$ are $\KK_0$-rational, there exists a $\KK_0$-rational $(n-2)$-element $U'\subset M'_1\cap M_2$.
  If $U' = U$ then $U$ is $\KK_0$-rational and we are done.

  Suppose that $U' \neq U$. Therefore $U' \not\subseteq M_1$, otherwise $U' = M_1\cap M'_1 = U$. As both $M_1$ and $U'$ are $\KK_0$-rational, their intersection $W := M_1\cap U'$ is $\KK_0$-rational. Note that $\dim(W) = n-3$, as one can see by noticing that $M_1\cap U' = (M_1\cap U')\cap U' = (M_1\cap M_2)\cap U'$  and recalling that $M_1\cap M_2$ is a hyperplane of $M_2$.

Consider the orthogonal $W^\perp$ of $W$ with respect to the form $q$ of Section~\ref{models2}. Taking equation~\eqref{forma quadratica} into account and recalling that $W$ is $\KK_0$-rational, we see that $W^\perp$ is a $\KK_0$-rational vector subspace of $V_{2n}(\KK)$. In fact $W^\perp$ is the span in $V_{2n}(\KK)$ of the orthogonal $W_0^\perp \subset V_{2n,E}(\KK_0)$ of
$W_0 := W\cap V_{2n, E}(\KK_0)$ with respect to the form $q_0$ induced by $q$ on $V_{2n,E}(\KK_0)$. All of the spaces $M_1, M_1', U, U'$ and $M_2$
contain $W$ and are totally singular, hence they are contained in $W^\perp$. Moreover, $M_1$ and $U'$ are $\KK_0$-rational.

As $W^\perp$ is $\KK_0$-rational, we can choose a basis $B = (w_1, \dots, w_{n+3})$ of $W^\perp$ formed by $\KK_0$-rational vectors. We can also assume that $w_1,\dots, w_{n-3}$ span $W$. As $B$ consists of $\KK_0$-rational vectors, a vector subspace of $W^\perp$ is $\KK_0$-rational with respect to $B$ if and only if it is $\KK_0$-rational with respect to $E$. Accordingly, $M_1$ and $U'$ are $\KK_0$-rational with respect to $B$ while $M'_1, U$ and $M_2$ are not.

We now switch to the quotient $W^\perp/W$, taking the cosets $\bar{w}_i := w_{n-3+i}+W$ for $i = 1, 2,\dots, 6$ to form a basis $\overline{B}$ of $W^\perp/W$. Since $W$ is totally singular, the form $q$ induces a quadratic form $\bar{q}$ on $W^\perp/W$. Let $\overline{\Delta}$ be the $D_3$-building associated to $\bar{q}$ and let $\overline{\Gamma}$ be its $(+,-)$--grassmannian. So, $(M_1/W, M_2/W)$ and $(M'_1/W, M_2/W)$ are points of the line $\ell_{U/W, M_2/W}$ of $\overline{\Gamma}$, with $U/W \neq U'/W$ and $U'/W\subset M'_1/W\cap M_2/W$. By the above, $M_1/W$ and $U'/W$ are $\KK_0$-rational while $M'_1/W, M_2/W$ and $U/W$ are not. We now switch from the $D_3$-building $\overline{\Delta}$ to the corresponding $A_3$-geometry, with elements of type $+$ and $-$ realized as points and planes of $\mathrm{PG}(V_4(\KK))$. In this new perspective, the above situation looks as follows: we have two distinct points $p$ and $p'$ (corresponding to $M_1/W$ and $M'_1/W$), two distinct lines $L$ and $L'$ (corresponding to $U/W$ and $U'/W)$ and a plane $S$ (corresponding to $M_2/W$). Both $p$ and $p'$ belong to $L$, $p'\in L'$ but $p\not \in L'$. Moreover, $S$ contains both $L$ and $L'$. Hence $S$ is spanned by $p$ and $L'$. However, $M_1/W$ and $U'/W$ are $\KK_0$-rational. Therefore, in view of Lemma~\ref{exterior}, both $p$ and $L'$ are $\KK_0$-rational with respect to a suitable basis $B'$ of $V_4(\KK)$. Hence $S$ is $\KK_0$-rational with respect to $B'$, since it is spanned by $p$ and $L'$. By exploiting Lemma~\ref{exterior} once again, we obtain that $M_2/W$ is $\KK_0$-rational with respect to $\overline{B}$. We have reached a final contradiction, which shows that the case we have been considering cannot occur.
\end{enumerate}
\end{enumerate}
The proof is complete.
\end{proof}

\begin{lemma}
\label{lemma32}
Let $V$ be a vector space over a division ring $\KK$ and $E = (e_1,\dots,e_n)$ a basis of $V$.
Let $\KK_0$ be a proper sub-division ring of $\KK$ and take $\eta\in\KK\setminus\KK_0.$
Suppose $S$ is a subspace of $V$ containing $e_1+e_2 \eta$. If $S$ is $\KK_0$-rational (with respect to $E$) then $e_1,e_2\in S.$
\end{lemma}
\begin{proof}
Following our conventions, we assume that $V$ is a right vector space. Let $V_0$ be the $\KK_0$-vector space of the $\KK_0$-rational vectors of $V$ (with respect to $E$). In order to avoid any confusion, we denote spans in $V$ by the symbol $\langle \dots \rangle_V$ and spans in $V_0$ by the symbol $\langle\dots \rangle_{V_0}$.

Assuming that $S$ is $\KK_0$-rational, let $(v_1,\dots v_k)$ be a basis of $S$ consisting of $\KK_0$-rational vectors and suppose that $e_1+ e_2 \eta\in S$. Then $\dim(S\cap \langle e_1,e_2\rangle_V)\geq 1.$
Note that the vector space $S_0 := \langle v_1,\dots,v_k\rangle_{V_0} = S\cap V_0$ has the same dimension as $S$. Thus,
since $\dim(S\cap \langle e_1,e_2\rangle_V)\geq 1,$ we also have $\dim(S_0\cap\langle e_1,e_2\rangle_{V_0})\geq 1$ by the well known
Grassmann dimension formula. It follows that there exists a non-zero vector $w\in S_0$ which is a linear combination $w= e_1 c_{1}+ e_2 c_{2}$ with $c_{1},c_{2}\in\KK_0$ and $(c_{1},c_{2})\neq (0,0)$. If either $c_1=0$ or $c_2=0$, then we are done. So, we can assume that $c_1 \neq 0 \neq c_2$.
Without loss of generality,  we can put $c_{1}=1$, so that  $w_1= e_1+ e_2 c_{2}$ with  $c_{2}\in\KK_0.$ Now we claim that there exists  $j_0\in \{1,\dots k\}$ such that ${v}_{j_0} = e_1 a_{1,j_0}+e_2 a_{2,j_0}+\cdots+  e_n a_{n,j_0}$ with $c_2 a_{1,j_0}\neq a_{2,j_0}$. By way of contradiction, suppose that for all $v_j\in \{v_1,\dots, v_k\}$  we have
\[v_j ~ = ~ e_1 a_{1,j}+e_2 a_{2,j}+e_3 a_{3,j}+\dots+e_n a_{n,j}\]
with $c_{2} a_{1,j}= a_{2,j}$, i.e. $(a_{1,j},a_{2,j})=(1,c_2) d_j$  for some $d_j\in \KK.$ This implies that for all vectors $v\in S$ we have
\[v ~= ~ v_1 \lambda_1+v_2 \lambda_2+\dots + v_k \lambda_k ~ = ~ e_1 (\sum_{i=1}^{k}\lambda_i+ \sum_{i=1}^{k}d_i)+e_1 c_{2}(\sum_{i=1}^{k}\lambda_i+ \sum_{i=1}^{k}d_i)+ u\]
with $u\in \langle e_3,\dots, e_n\rangle, \lambda_i, d_i\in \KK$ and $c_{2}\in\KK_0.$ In particular, taking $v=e_1+ e_2 \eta\in S$ we have
$\sum_{i=1}^{k}\lambda_i+ \sum_{i=1}^{k}d_i=1$ and $c_{2} (\sum_{i=1}^{k}\lambda_i+ \sum_{i=1}^{k}d_i)=\eta$, forcing $\eta=c_{2}\in \KK_0$ which is a contradiction. The claim is proved.

Consider the ordered pair $(w, v_{j_0})$. As $w, v_{j_0}\in S_0$, we can complete this pair to an ordered basis $B$ of $S_0$ by choosing $k-2$ suitable vectors from the $k-1$ vectors in $\{v_1,\dots,v_k\}\setminus \{v_{j_0}\}$. Without getting out of $V_0$, we can now apply a full Gaussian reduction to the sequence of vectors of $B$ to obtain another basis $(v'_1,\dots,v'_k)$ of $S_0$ such that the $(n\times k)$-matrix $M$ of the coefficients of the vectors $v'_1,\dots, v'_k$ with respect to $e_1,\dots,e_n$ is in Column Reduced Echelon Form. (Note that, according to our convention to deal with right vector spaces, vectors should be represented as columns.) By construction, the matrix $M$ contains the identity matrix $I_k$ as a minor. Up to a permutation of the vectors $e_3,\dots,e_n$ we can suppose that this minor encompasses the first $k$ rows of the matrix $M$. The remaining $n-k$ rows form an $((n-k)\times k)$-matrix
\[N ~ =  ~ (b_{k+i,j})_{i, j ~= 1}^{n-k,k} \]
with entries $b_{k+i, j}\in \KK_0$. However $e_1+ e_2 \eta\in S = \langle S_0\rangle_V =  \langle v_1', v_2',\dots,v_k'\rangle_V$. Hence there exist $\alpha_1,\dots,\alpha_k\in\KK$ such that $e_1+ e_2 \eta=v_1' \alpha_1+v_2' \alpha_2+\dots+ v_k' \alpha_k$. For every $i=1,\dots, k$ we have $v_i'=e_i+\sum_{j=k+1}^{n}e_j b_{j,i}$. Therefore
\[ e_1+ e_2 \eta=\sum_{i=1}^{k}e_i \alpha_i +e_{k+1} (\sum_{j=1}^{k}b_{k+1,j}\alpha_j)+ e_{k+2} (\sum_{j=1}^{k}b_{k+2.j}\alpha_j)+\dots+e_{n} (\sum_{j=1}^{k}b_{n,j}\alpha_j),\]
which implies $\alpha_1=1, \alpha_2=\eta$, $\alpha_3=\alpha_4=\dots =\alpha_k=0$ and
\[\sum_{j=1}^{k}b_{k+1,j}\alpha_j ~=~ \sum_{j=1}^{k}b_{k+2,j}\alpha_j ~=~\dots ~ =~ \sum_{j=1}^{k}b_{n,j}\alpha_j ~ = ~ 0.\]
It follows that $e_1+e_2\eta=v_1'+v_2'\eta$,  whence $(e_1-v_1')=(v_2'-e_2)\eta$. However,
\[ (e_1-v_1') ~ = ~ \sum_{i=k+1}^n e_i(-b_{i,1}),\quad  (v_2'-e_2) ~ = ~\sum_{i=k+1}^n e_i(b_{i,2}\eta), \]
whence $-b_{i,1}=b_{i,2}\eta$ for all $i\geq k+1$. Since $b_{i,j}\in\KK_0$ for all $i,j$ and the elements $1,\eta\in\KK$ are linearly independent
over $\KK_0$, it follows that $b_{i,1}=b_{i,2}=0$ for all $i\geq k+1$. So, $v_1'=e_1$ and $v_2'=e_2$. Since $v_1',v_2'\in S$, we obtain $e_1,e_2\in S$, which proves the lemma.
 \end{proof}

\begin{lemma}
  \label{Omega-proper}
If $\Gamma(\KK)$ is as in the hypotheses of Theorem {\rm \ref{m-th3}} then not all points of $\Gamma(\KK)$ belong to  $\Omega_{\KK_0}(\Gamma(\KK))$.
\end{lemma}
\begin{proof}
 We first consider the case $\Gamma(\KK)=\Gr_{1,n}(A_n(\KK))$. Pick $\eta\in \KK\setminus \KK_0$. With $E = (e_1,\dots, e_{n+1})$ as in Section \ref{rational1}, put $p = \langle e_1+ e_2\eta\rangle$  and $H=\langle e_1+e_2\eta, e_3,\dots,e_{n},e_{n+1}\rangle$. (Needless to say, the symbol $\langle\dots \rangle$ refers to spans in $V_{n+1}(\KK)$.) The flag  $(p, H)$ is a point of $\Gr_{1,n}(A_n(\KK))$. Let $S$ be a subspace of $V_{n+1}(\KK)$ such that $p\subseteq S\subseteq H.$ Any such subspace contains the vector $e_1+e_2\eta$ but neither $e_1$ nor $e_2$.
Hence $S$ cannot be $\KK_0$-rational, by Lemma~\ref{lemma32}. Consequently, $(p, H)\not\in \Omega_{\KK_0}(\Gamma(\KK)).$

The case $\Gamma(\KK)=\Gr_{{1,+,-}}(D_n(\KK))$ is entirely analogous. With $E = (e_1, e_2,\dots, e_{2n})$ as in Section \ref{rational2} and $\eta$ as above, put $p=\langle e_1+e_3\eta\rangle$, $M_1=\langle e_1+e_3\eta,e_2\eta-e_4,e_5,e_7,\dots , e_{2n-1}\rangle$ and $M_2=\langle e_1+e_3\eta,e_2\eta- e_4,e_5,e_7,\dots, e_{2n-1}\rangle$. Taking equation~\eqref{forma quadratica} into account, it is straightforward to see that $p, M_1$ and $M_2$ belong to $D_n(\KK)$. It is also easy to see that they form a $(1, +, -)$--flag of $D_n(\KK)$, namely a point of $\Gamma(\KK)$. Clearly, none of $p$, $M_1$ or $M_2$ is $\KK_0$-rational and Lemma \ref{lemma32} implies that none of the subspaces contained in $M_1\cap M_2$ and containing $p$ can be $\KK_0$-rational. Hence $(p, M_1, M_2)\not\in \Omega_{\KK_0}(\Gamma(\KK)).$

When $\Gamma(\KK)=\Gr_{{1,-}}(D_n(\KK))$ we can consider the flag $(p,M)$ where $p=\langle e_1+e_3\eta\rangle$ and  $M=\langle e_1+e_3\eta, e_2\eta-e_4, e_5, e_7, \dots, e_{2n-1}\rangle$. The subspace $M$ is $n$-dimensional and totally singular for $q$. We can also assume to have chosen the signs $+$ and $-$ in such a way that $\fS^-$ is indeed the class which $M$ belongs to. So, $(p,M)$ is a point of $\Gamma(\KK)$. Once again, by Lemma \ref{lemma32} we see that $(p,M)\not\in \Omega_{\KK_0}(\Gamma(\KK)).$

Finally, let $\Gamma(\KK) = \Gr_{+,-}(D_n(\KK))$. In view of Lemma~\ref{exterior}, if $n = 3$ we are back to $A_3$. So, assume $n > 3$. With $\eta\in \KK\setminus \KK_0$ and $E = (e_1,\dots, e_{2n})$ as in Section \ref{rational2}, put
\[\begin{array}{ccl}
M_1 & := & \langle e_1+e_3, e_2 - e_4, e_5+e_7\eta, e_6\eta - e_8, e_{10}, e_{12},\dots, e_{2n}\rangle, \\
 M_2 & := & \langle e_1+ e_4, e_2 - e_3, e_5+e_7\eta, e_6\eta - e_8, e_{10}, e_{12},\dots, e_{2n}\rangle.
\end{array}\]
Then $M_1$ and $M_2$ are $n$-dimensional totally singular subspaces of $D_n(\KK)$ but neither of them is $\KK_0$-rational. Moreover $M_1\cap M_2 = \langle e_1-e_2 + e_3 + e_4, e_5+e_7\eta, e_6\eta - e_8, e_{10}, e_{12},\dots, e_{2n}\rangle$.
Hence $\{M_1, M_2\}$ is a $\{+,-\}$--flag of $D_n(\KK)$, necessarily not $\KK_0$-rational, since neither $M_1$ nor $M_2$ is $\KK_0$-rational. Accordingly, $M_1\cap M_2$ is not $\KK_0$-rational. In fact all $\KK_0$-rational subspaces of $M_1\cap M_2$ are contained in $\langle e_1-e_2+e_3+ e_4, e_{10}, e_{12},\dots, e_{2n}\rangle$, which is $(n-3)$-dimensional. Their dimensions are too small for them to split $(+,-)$. Therefore $(M_1, M_2)\not\in \Omega_{\KK_0}(\Gamma(\KK)).$
\end{proof}
Lemmas \ref{Omega-geo} and \ref{Omega-proper} yield Theorem \ref{m-th3}.

\subsection{Proof of Corollary \ref{vecchio}}\label{subtle}
As already remarked in Section \ref{models2}, the function $\iota$ that maps every $(n-1)$-element of $B^+_n(\KK)$ onto the pair of $n$-elements containing it is an isomorphism from $\Gr_{n-1}(B^+_n(\KK))$ to $\Gr_{+,-}(D_n(\KK))$. We know from Theorem~\ref{m-th1} that if $\KK_0 < \KK$ then $\Gr_{+,-}(D_n(\KK_0))$ spans a proper subspace of $\Gr_{+,-}(D_n(\KK))$. In order to show that the same holds for $\Gr_{n-1}(B^+_n(\KK_0))$ and $\Gr_{n-1}(B^+_n(\KK))$, as claimed in Corollary \ref{vecchio}, we only need to prove the following:

\begin{proposition}
The isomorphism $\iota$ maps the subgeometry $\Gr_{n-1}(B_n^+(\KK_0))$ of $\Gr_{n-1}(B^+_n(\KK))$ onto the subgeometry $\Gr_{+,-}(D_n(\KK_0))$ of $\Gr_{+,-}(D_n(\KK))$.
\end{proposition}
\begin{proof}
It goes without saying that both $\Gr_{n-1}(B^+_n(\KK_0)) = \Gr_{n-1,E}(B^+_n(\KK_0))$ and $\Gr_{+,-}(D_n(\KK_0)) = \Gr_{+,-,E}(D_n(\KK_0))$ for the same basis $E$ of $V_{2n}(\KK)$, chosen as in Section \ref{rational2}.

Let $U = M_1\cap M_2$ for a $(+,-)$--flag $(M_1,M_2)$ of $D_n(\KK)$. If both $M_1$ and $M_2$ are $\KK_0$-rational then $U$ is $\KK_0$-rational,
by Lemma~\ref{rational An 2}. Conversely, let $U$ be $\KK_0$-rational. Let $M'_1$ and $M'_2$ be the two $n$-elements of $B^+_{n, E}(\KK_0)$ containing $U$. Then $M'_1$ and $M'_2$ are $\KK_0$-rational, as they belong to $B^+_{n,E}(\KK_0)$. However, they are the only two $n$-elements on $U$. Hence $\{M'_1, M'_2\} = \{M_1, M_2\}$. Accordingly, $M_1$ and $M_2$ are $\KK_0$-rational.
\end{proof}

\subsection{Proof of Theorem \ref{non-interval}}\label{Sec 3 bis}
We are not going to give a detailed proof of this theorem. We will only offer a sketch of it, leaving the details to reader.

As stated since the beginning of this section, $\KK_0$ is a proper sub-division ring of $\KK$ and $\Gamma(\KK) = \Gr_J(X_n(\KK))$, where $X_n$ stands for $A_n$ or $D_n$. According to the hypotheses of Theorem \ref{non-interval}, we assume that $J$ is not connected.

Suppose firstly that $J$ contains two types $j_1$ and $j_2$, with $j_1, j_2\leq n-2$ when $X_n = D_n$, such that $j_1+1 < j_2$ and $i\not\in J$ for every type $i \in \{j_1+1, j_1+2,\dots, j_2-1\}$. We say that a $J$-flag $F$ of $X_n(\KK)$ (point of $\Gamma(\KK)$) is {\em nearly} $\KK_0$-{\em rational at} $(j_1, j_2)$ if there exists a $\KK_0$-rational element $X$ of $X_n(\KK)$ incident to $F$ and such that $j_1 \leq \dim(X)\leq j_2$. Let $\Omega_{\KK_0,j_1,j_2}(\Gamma(\KK))$ be the set of $J$-flags which are nearly $\KK_0$-rational at $(j_1, j_2)$. Using the same argument as in \ref{omega-geo:1}) in the proof of Lemma~\ref{Omega-geo}, with the roles of $1$ and $n$ respectively taken by $j_1$ and $j_2$ we see that
$\Omega_{\KK_0,j_1,j_2}(\Gamma(\KK))$ is a subspace of $\Gamma(\KK)$. Next, by an argument similar to that used for $\Gr_{1,n}(A_n(\KK))$ in the proof of Lemma~\ref{Omega-proper}, we obtain that $\Omega_{\KK_0,j_1, j_2}(\Gamma(\KK)) \neq \Gamma(\KK)$, namely $\Omega_{\KK_0,j_1, j_2}(\Gamma(\KK))$ is a proper subspace of $\Gamma(\KK)$. However $\Gamma(\KK_0) := \Gr_J(X_n(\KK_0))$ is contained in $\Omega_{\KK_0,j_1, j_2}(\Gamma(\KK))$. Hence $\Gamma(\KK_0)$ spans a proper subspace of $\Gamma(\KK)$, as stated in Theorem~\ref{non-interval}.

Two more possibilities remain to examine, which are not considered in
Theorem~\ref{m-th1}, namely $X_n(\KK)=D_n(\KK)$ and $J$ as follows:
\begin{enumerate}
\item
$J = \{j, j+1,\dots, j+k\} \cup \{+,-\}$ for $j \geq 1$, $j+k < n-2$ and either $j > 1$ or $k > 0$. In this case we can use the same arguments as for $J = \{1, +, -\}$ in the proof of Theorem \ref{m-th1}, with $j+k$ playing the role of $1$.
\item
$J = \{j, j+1,\dots, j+k\} \cup \{-\}$ or $J = \{j, j+1,\dots, j+k\} \cup \{+\}$, for $j \geq 1$, $j+k < n-2$ and either $j > 1$ or $k > 0$. The arguments used for $J = \{1, -\}$ work for this case as well, with $1$ replaced by $j+k$.
\end{enumerate}

\section{Proof of Lemma \ref{cor-inf} and Theorem \ref{m-th2}}\label{Nuovo}

\subsection{Proof of Lemma~\ref{cor-inf}}

Assume that $J$ is non-connected and $\KK$ is not finitely generated. Let $S$ be a finite set of points of $\Gamma(\KK) = \Gr_J(X_n(\KK))$, where $X_n$ stands for $A_n$ or $D_n$. Each element $F$ of $S$ is a $J$-flag $F = \{U_1,U_2,\dots, U_t\}$ of vector subspaces $U_i$ of $V_N(\KK)$, where $t := |J|$ and $N$ is $n+1$ or $2n$ according as $X_n$ is $A_n$ or $D_n$. Fix a basis $B_{i,F}$ for each of the  vector subspaces $U_i \in F$ and each $F\in S$ and let $C(S)$ be the set of all the coordinates of the vectors of $\cup_{F\in S}\cup_{i=1}^t B_{i,F}$ with respect to a given basis of $V_N(\KK)$
(chosen as in Section~\ref{rational2} when $X_n = D_n$).

Since $S$ is finite, $C(S)$ is finite as well; in fact $|C(S)| \leq t\cdot N\cdot |S|$. Therefore, and since $\KK$ is not finitely generated, $C(S)$ generates a proper sub-division ring $\KK_0$ of $\KK$. Then $\Gamma(\KK_0) := \Gr_J(X_n(\KK_0))$ spans a proper subspace of $\Gamma(\KK)$, by Theorem \ref{non-interval}. Obviously, $S$ is contained in $\Gamma(\KK_0)$. Hence $S$ spans a proper subspace of $\Gamma(\KK)$. Thus we have proved that no finite subset of $\Gamma(\KK)$ generates $\Gamma(\KK)$, as claimed in Lemma \ref{cor-inf}.

\subsection{Proof of Theorem \ref{m-th2}}

Put $\Gamma := \Gr_{1,n}(A_n(\overline{\mathbb{F}}_p))$. We have $\mathrm{gr}(\Gamma)=\infty$ by Lemma~\ref{cor-inf}, since $\overline{\FF}_p$ is not finitely generated. The geometry $\Gamma$ admits a (full) projective embedding of dimension $(n+1)^2-1$, namely the embedding $e_{\mathrm{Lie}}$ mentioned in Remark \ref{Rem1.6}. Therefore $\mathrm{er}(\Gamma)\geq (n+1)^2-1$.

By way of contradiction, suppose that $\mathrm{er}(\Gamma) > (n+1)^2$. Then $\Gamma$ admits a (full) projective embedding $e:\Gamma\rightarrow \mathrm{PG}(V)$ of dimension $\dim(e) \geq (n+1)^2+1$. Consequently, there exists a set $S$ of $(n+1)^2+1$ points of $\Gamma$ such that $\cup_{x\in S}e(x)\subset V$ spans a subspace $V_S$ of $V$ of dimension $\dim(V_S) = (n+1)^2+1$.

Every point $x\in S$ is a point-hyperplane flag $(p_x, H_x)$ of $A_n(\overline{\FF}_p)$. For every $x\in S$ we choose a non-zero vector $v_x\in p_x$ and a basis $B_x$ of $H_x$. Chosen a basis $E$ of $V_{n+1}(\overline{\FF}_p)$, let $C(S)$ be the set of all elements of $\overline{\FF}_p$ which occur as coordinates (with respect to $E$) of either $v_x$ or a vector of $B_x$, for $x\in S$. The set $C(S)$ is finite. Hence it generates a finite subfield $\LL$ of $\overline{\FF}_p$. Every point $x\in S$ is obviously $\LL$-rational. Therefore $S \subset \Gamma_\LL := \Gr_{1,n}(A_n(\LL)) \subset \Gamma$.

Let $V_\LL$ be the subspace of $V$ corresponding to the span of $e(\Gamma_\LL)$. Clearly $V_\LL \supseteq V_S$. Hence $\dim(V_\LL) \geq \dim(V_S) = (n+1)^2+1$. The restriction $e_\LL$ of $e$ to $\Gamma_\LL$ is a lax embedding of $\Gamma_\LL$ in $\mathrm{PG}(V_\LL)$. As noticed in Remark \ref{lax}, inequality \eqref{gen emb rk 1} holds for lax embeddings too. Therefore $\Gamma_\LL$ has generating rank $\gr(\Gamma_\LL) \geq \dim(e_\LL) = \dim(V_\LL) > (n+1)^2$.

On the other hand, the field $\LL$ is a simple extension of the prime field $\FF_p$ and $\Gr_{1,n}(A_n(\FF_p))$ has generating rank equal to $(n+1)^2-1$, by Cooperstein \cite{C98b}. Therefore $\mathrm{gr}(\Gamma_\LL)\leq (n+1)^2$ by Blok and Pasini \cite[Corollary 4.8]{BPa01}. We have reached a contradiction. Consequently, $\mathrm{er}(\Gamma) \leq (n+1)^2$. The proof of Theorem \ref{m-th2} is complete.

 \end{document}